\documentclass[11pt]{amsart}

\usepackage{amssymb}
\usepackage[all]{xy}
\usepackage[colorlinks=true, urlcolor=rltblue, citecolor=drkgreen, linkcolor=drkred] {hyperref}
\usepackage{xcolor}
\usepackage{pdfsync}
\definecolor{rltblue}{rgb}{0,0,0.4}
\definecolor{drkred}{rgb}{0.6,0,0}
\definecolor{drkgreen}{rgb}{0,0.4,0}
\usepackage{graphicx}
\usepackage[mathscr]{eucal}

\usepackage{thmtools}
\usepackage{thm-restate}

\usepackage{lmodern}
\usepackage[capitalize]{cleveref}
\usepackage{amssymb,amsmath,amsthm}
\usepackage{mathtools, MnSymbol}
\usepackage{setspace}
\usepackage{tikz-cd}
\usepackage[T1]{fontenc}
\usepackage[utf8]{inputenc}
\usepackage{textcomp} 
\usepackage{stmaryrd}
\usepackage{datetime}
\usepackage{todonotes}
\usepackage{mdframed}
\usepackage[all]{xy}

\providecommand{\tightlist}{%
	\setlength{\itemsep}{0pt}\setlength{\parskip}{0pt}}
\usepackage{thmtools}
\usepackage{enumitem}

\declaretheorem[numberwithin=section]{theorem}
\declaretheorem[sibling=theorem]{lemma}
\declaretheorem[sibling=theorem]{proposition}
\declaretheorem[sibling=theorem]{corollary}
\declaretheorem[sibling=theorem]{definition}
\declaretheorem[sibling=theorem]{question}
\declaretheorem[numberwithin=theorem]{claim}
\declaretheorem[numberwithin=theorem]{sublemma}

\newcommand{\B}{\mathcal{B}}

\newcommand{\SR}{\text{SR}}

\renewcommand{\L}{\mathcal{L}}
\newcommand{\mc}[1]{\mathcal{#1}}

\renewcommand{\phi}{\varphi}
\newcommand{\bigwwedge}{%
	\mathop{
		\mathchoice{\bigwedge\mkern-15mu\bigwedge}
		{\bigwedge\mkern-12.5mu\bigwedge}
		{\bigwedge\mkern-12.5mu\bigwedge}
		{\bigwedge\mkern-11mu\bigwedge}
	}
}

\newcommand{\bigvvee}{%
	\mathop{
		\mathchoice{\bigvee\mkern-15mu\bigvee}
		{\bigvee\mkern-12.5mu\bigvee}
		{\bigvee\mkern-12.5mu\bigvee}
		{\bigvee\mkern-11mu\bigvee}
	}
}
\newmdtheoremenv[backgroundcolor=cyan]{theorem-prove}{Theorem}[theorem]
\newmdtheoremenv[backgroundcolor=cyan]{lemma-prove}{Lemma}[theorem]
\newmdtheoremenv[backgroundcolor=cyan]{proposition-prove}{Proposition}[theorem]

\newmdtheoremenv[backgroundcolor=yellow!40]{theorem-check}{Theorem}[theorem]
\newmdtheoremenv[backgroundcolor=yellow!40]{lemma-check}{Lemma}[theorem]
\newmdtheoremenv[backgroundcolor=yellow!40]{proposition-check}{Proposition}[theorem]

\DeclareMathOperator{\Mod}{Mod}
\DeclareMathOperator{\Image}{Im}

\def\hbar{{\bar{h}}}

\def\om{\omega}

\def\B{\mathcal{B}}

\def\K{\mathcal K}
\def\L{\mathcal L}

\def\Lomom{\L_{\om_1,\om}}

\newcommand{\Res}{\upharpoonright}

\newcommand{\newmacro}[1]{\mathfrak{#1}}
\newcommand{\A}{\newmacro{A}}
\newcommand{\E}{\newmacro{E}}
\newcommand{\vwE}{\overline{\E}}
\newcommand{\vwA}{\overline{\A}}

\newtheorem{thm}{Theorem}[section]

\theoremstyle{remark}

\newtheorem{remark}[thm]{Remark}

\newcommand{\bfPi}{\boldsymbol{\Pi}}
\newcommand{\bfSigma}{\boldsymbol{\Sigma}}

\def\and{\mathrel{\&}}

\begin{document}
	
	\title{Scott Spectral Gaps are Bounded for Linear Orderings}
	\author{David Gonzalez}
	\address{Department of Mathematics, University of California, 970 Evans Hall, Berkeley, CA, 94720}
	\email{david\_gonzalez@berkeley.edu}
	\author{Matthew Harrison-Trainor}
	\address{Department of Mathematics, Statistics, and Computer Science, University of Illinois Chicago, 851 S Morgan St, Chicago, IL, 60607}
	\email{mht@uic.edu}
	\thanks{The first author is the corresponding author. The second author was supported by a Sloan Research Fellowship and by the National Science Foundation under Grant No.\ \mbox{DMS-2419591}. We thank Garrett Ervin for some discussion on results of Lindenbaum and Tarski.}

\begin{abstract}
	We demonstrate that any $\Pi_\alpha$ sentence of the infinitary logic $\mathcal{L}_{\omega_1 \omega}$ extending the theory of linear orderings has a model with a $\Pi_{\alpha+4}$ Scott sentence and hence of Scott rank at most $\alpha+3$.
	In other words, the gap between the complexity of the theory and the complexity of the simplest model is always bounded by $4$.
	This contrasts the situation with general structures where for any $\alpha$ there is a $\Pi_2$ sentence all of whose models have Scott rank $\alpha$. We also give new lower bounds, though there remains a small gap between our lower and upper bounds: For most (but not all) $\alpha$, we construct a $\Pi_\alpha$ sentence extending the theory of linear orderings such that no models have a $\Sigma_{\alpha+2}$ Scott sentence and hence no models have Scott rank less than or equal to $\alpha$.
\end{abstract}


	\maketitle
	
	\section{Introduction}
	
	Consider some class of mathematical structures, such as linear orders (which will be the main class of interest in this paper) or graphs, groups, or Boolean algebras, or even subclasses of these such as the dense linear orders or the connected graphs. Each class has a \textit{language} $\mc{L}$, which consists of specified functions, relations, and constant symbols---for linear orders, this is a binary relation symbol ``$\leq$''---and certain properties that these must satisfy. These properties should be specific in some particular formal language, which for us will be the infinitary logic $\mc{L}_{\omega_1 \omega}$.  This infinitary logic allows us to use countable infinite conjunctions and disjunctions, so that we 
	can express properties like the connectedness of a graph $(G,E)$ as ``for all $u,v$ there is $n$ and $w_1,\ldots,w_n$ such that $u E w_1 E w_2 E \cdots E w_n E v$''. Connectedness cannot be expressed in the standard first-order logic, nor can other properties of interest such as whether a group is finitely generated.
	
	In this paper, we consider the following question: Given a class of linear orders, how does the complexity of the individual linear orders in the class compare to the complexity of the description of the class? In particular, given a relatively simply described class of linear orders, must it contain a relatively simple linear order? We show that this is the case, which contrasts with general structures such as graphs where there can be simply described classes of graphs all of which are complicated. To measure complexity, we use the \textit{Scott analysis} to which we now give a short introduction before stating our results.
	
	We will be interested in countable structures, which in a language $\mc{L}$ form a Polish topological space $\Mod(\mc{L})$. Each countable $\mc{L}$-structure is isomorphic to one whose domain is $\mathbb{N}$, and so we can identify, e.g., a linear order $\mc{A} = (\mathbb{N},\precsim)$ with the infinite binary string
	\[ \sigma_{\mc{A}}(i) = \begin{cases}
		1 & a_i \precsim b_i \\
		0 & \text{otherwise}
	\end{cases} \]
	where $(a_i,b_i)$ is some fixed listing of the pairs of natural numbers. The topology on $\Mod(\mc{L})$ is the same, via this identification, as the topology on Cantor space $2^\mathbb{N}$. For other languages $\Mod(\mc{L})$ may be of a different form, such as Baire space $\mathbb{N}^\mathbb{N}$. There is also a continuous action of $S_\infty$, the permutation group of $\mathbb{N}$, on $\Mod(\mc{L})$ by permuting the underlying domain of the structure (and hence induces isomorphisms). A set $X \subseteq \Mod(\mc{L})$ is invariant if whenever $\mc{A}$ and $\mc{B}$ are isomorphic, $\mc{A} \in X$ if and only if $\mc{B} \in X$, or equivalently, if $X$ is invariant under the action of $S_\infty$. Lopez-Escobar \cite{Lop66} showed that the infinitary logic $\mc{L}_{\omega_1 \omega}$ is a natural logic to consider in this context, because a subset of $\Mod(\mc{L})$ is definable by a sentence of $\mc{L}_{\omega_1 \omega}$ if and only if it is invariant and Borel.
	
	Fixing a particular structure $\mc{A}$, consider the set $\text{Copies}(\mc{A}) = \{ \mc{B} \in \Mod(\mc{L}): \mc{B} \cong \mc{A}\}$ of isomorphic copies of $\mc{A}$. Naively, this is analytic. A surprising result of Scott \cite{Sco65} is that it is always Borel, and so by the Lopez-Escobar theorem there is a sentence $\varphi$ of the infinitary logic $\mc{L}_{\omega_1 \omega}$ that characterizes $\mc{A}$ up to isomorphism among countable structures: For any countable structure $\mc{B}$, $\mc{B} \cong \mc{A}$ if and only if $\mc{B} \models \varphi$.
	Such a sentence is called a Scott sentence for $\mc{A}$. From the dynamic perspective, Scott showed that all of the orbits of the action of $S_\infty$ on $\Mod(\mc{L})$ are Borel.
	
	The Borel sets can be stratified into the Borel hierarchy: the open sets are called $\bfSigma^0_1$, and the closed sets $\bfPi^0_1$, while the $\bfSigma^0_\alpha$ sets are countable unions of $\bfPi^0_\beta$ sets for $\beta < \alpha$ and the $\bfPi^0_\alpha$ sets are the countable intersections of $\bfSigma^0_\beta$ sets for $\beta < \alpha$. Similarly, the formulas of $\mc{L}_{\omega_1 \omega}$ are stratified in terms of their complexity as measured by counting the number of alternations between existential and universal quantifiers.
	\begin{itemize}
		\item A formula is $\Sigma_0$ and $\Pi_0$ if it is finitary quantifier-free.
		\item A formula is $\Sigma_{\alpha}$ if it is of the form
		\[ \bigdoublevee_i \exists \bar{x}_i \psi_i(\bar{x}_i)\]
		where each $\psi_i$ is $\Pi_{\beta}$ for some $\beta < \alpha$.
		\item A formula is $\Pi_{\alpha}$ if it is of the form
		\[ \bigdoublewedge_i \forall \bar{x}_i \psi_i(\bar{x}_i)\]
		where each $\psi_i$ is $\Pi_{\beta}$ for some $\beta < \alpha$.
	\end{itemize}
	By Vaught's version of the Lopez-Escobar theorem \cite{Vaught}, an invariant set is Borel $\bfSigma^0_\alpha$ if and only if it is defined by a $\Sigma^0_\alpha$ sentence of $\mc{L}_{\omega_1 \omega}$, and similarly for $\bfPi^0_\alpha$ and $\Pi_\alpha$.
	
	To each structure $\mc{A}$ we can assign an ordinal-valued Scott rank, which measures the least complexity of a Scott sentence for $\mc{A}$, essentially, the complexity of describing $\mc{A}$ up to isomorphism.
	There are a number of different non-equivalent definitions of Scott rank, which are technically different but coarsely the same. 
	The definition we will use is due to Montalb\'an  in \cite{MonSR}, but the best reference is in Montalb\'an's book \cite{MBook}.
	The (parameterless) Scott rank, $\SR(\mc{A})$, is the least $\alpha$ such that $\mc{A}$ has a $\Pi_{\alpha+1}$ Scott sentence, or equivalently the least $\alpha$ such that $\text{Copies}(\mc{A})$ is $\bfPi^0_{\alpha+1}$.  
	Not only does this definition of Scott rank have a tight connection to Scott sentences, it is also a robust measure of complexity, as there are also internal characterizations in terms of the definability of orbits and the difficulty of computing isomorphisms between copies of the structures \cite{MonSR}.
	
	Consider some reasonably defined class of structures, i.e., one which is an invariant Borel subset of $\Mod(\mc{L})$, or equivalently, a class of structures defined by an infinitary sentence $T$ which we think of as a theory defining its class of models. The notion of Scott rank (and the related back-and-forth relations) does not just measure the complexity of individual structures, but also gives insight into the properties of infinitary theories. (For example, Vaught's conjecture is really about the structure of the tree of back-and-forth types of infinitary theories.) In this paper, given an infinitary sentence $T$, we can consider the set of Scott ranks of the models of $T$.
	\begin{definition}
	Let $T$ be a $\Lomom$ sentence. The \emph{Scott spectrum} of $T$ is the set
	\[SS(T)=\{\alpha\in\omega_1 \;\; | \;\; \text{there is $\mc{A} \models T$ of Scott rank $\alpha$}\}.\]
	\end{definition} 
	\noindent The Scott spectrum is the set of complexities of structures in the class defined by $T$, or in the dynamic interpretation, the set of Borel complexities of the individual orbits making up the class of models of $T$. The interesting questions about Scott spectra are generally about relating the complexity of $T$, the definition of the class, to the complexities of its models. In this paper, we will focus on the question of gaps in the Scott spectrum, particularly the gap between the complexity of $T$ and the minimal complexity of a model of $T$. Must $T$ have a model that is around the same complexity as $T$? Or is it possible for all of the models of $T$ to be much more complex than $T$?

	The history of such questions starts in 2013 when Montalb\'an was developing his notion of Scott rank and its many equivalent characterizations. He asked the following question at the BIRS 2013 Workshop in Computable Model Theory.
	\begin{question}\label{question:BIRS}
	If $T$ is a $\Pi_2$ sentence, must $T$ have a model with a $\Pi_{3}$ Scott sentence and hence of Scott rank $2$ or less?
	\end{question}
	\noindent More generally, one can ask whether if $T$ is a $\Pi_\alpha$ sentence, must every model of $T$ have Scott rank $\alpha$ or less? This was surprisingly answered negatively in a very strong way by Harrison-Trainor \cite{HTScott}.
	
	\begin{theorem}[Harrison-Trainor]\label{thm:BigGaps}
		For any countable ordinal $\alpha$, there is a satisfiable $\Pi_2$ sentence $T$ such that all of the models of $T$ have Scott rank $\geq \alpha$.
	\end{theorem}
	
	\noindent Thus it is possible for a class of structures with a very simple description to have only very complex models. This paper is about the same question in the context of linear orders.
	
	\begin{question}\label{question:SpectralGap}
		Given $\alpha$, does there exist $\beta$ such that if $T$ is a $\Pi_\alpha$ sentence expanding the theory of linear orders, then $T$ has a model of Scott rank at most $\beta$? If so, what is the least such $\beta$?
	\end{question}
	
	\noindent Linear orders are a particularly interesting class of examples in which to ask this question. First of all, the isomorphism relation for linear orders is Borel complete in the sense of Friedman and Stanley \cite{FS89}, so there are linear orders of arbitrarily high Scott ranks. On the other hand, linear orders are not universal, meaning that they cannot bi-interpret any other class of structures (see \cite{HKSS,HTMM,HTMMM} and VI.3.2 of \cite{MBook}). Because they are not universal, whether linear orders have unbounded Scott gaps does not immediately follow from Theorem \ref{thm:BigGaps} (as it would for universal classes such as the class of graphs). Moreover, the proof of Theorem \ref{thm:BigGaps} uses complicated relationships between large tuples of elements, whereas it has become increasingly clear that in the back-and-forth analysis of linear orders only relationships between pairs of elements matter (see \cite{GM23,GR23,GHTH}). We prove:
	
	\begin{theorem}\label{thm:SmallGaps}
		Given a satisfiable $\Pi_{\alpha}$ sentence $T$ of linear orders, there is a linear order $\mc{B} \models T$ with a $\Pi_{\alpha + 4}$ Scott sentence and hence Scott rank at most $\alpha + 3$.
	\end{theorem}
	
	\noindent This is the first non-trivial example in which such a theorem is known. This theorem gives another perspective on why linear orderings are not universal:
	Their Scott spectra are not as rich as those of a general structure.
	This means that there can be no sufficiently well-behaved interpretation of general structures into linear orderings, as every such interpretation must shift the Scott spectra enough to collapse large gaps.
	
	\medskip
	
	The study of gaps in Scott spectra has connections with the theory of Borel reducibility for classes of countable structures as introduced by Friedman and Stanley \cite{FS89}. Let $\mathbb{K}$ and $\mathbb{L}$ be Borel classes of countable structures with domain $\mathbb{N}$ in the Polish spaces $\Mod(\mc{L})$ for their respective languages. We say that $\mathbb{K}$ is Borel reducible to $\mathbb{L}$ if there is a Borel map $\Phi: \mathbb{K} \to \mathbb{L}$ such that
	\[ \mc{A} \cong \mc{B} \Longleftrightarrow \Phi(\mc{A}) \cong \Phi(\mc{B}).\]
	In particular, we say that $\mathbb{L}$ is \textit{Borel complete} if every Borel class of structures (or, sufficiently, graphs) is Borel reducible to $\mathbb{L}$. Friedman and Stanley showed that various classes of structures are Borel complete, including linear orders, trees, and graphs. Camerlo and Gao \cite{CamerloGao} showed that Boolean algebras are Borel complete, and recently Paolini and Shelah \cite{PaoliniShelah}, with another proof by Laskowski and Ulrich \cite{LaskowskiUlrich}, showed that torsion-free abelian groups are Borel complete.
	
	We can strengthen Borel reducibility to ask that given a Borel set of the domain, the closure under isomorphism equivalence of the image is Borel, or equivalently, that the closure under isomorphism of the range of the reduction is Borel. We call such a reduction \textit{faithful}. If $\mathbb{L}$ is Borel complete via a faithful Borel reduction, we say that $\mc{L}$ is faithfully Borel complete. If $\mc{L}$ is faithfully Borel complete, then Vaught's conjecture for is true for $\mathbb{L}$ if and only if Vaught's conjecture is true in general---indeed this was one of Friedman and Stanley's motivations for introducing Borel reducibility. Steel \cite{Steel} proved Vaught's conjecture for linear orders, and later Gao \cite{Gao01} showed that linear orders are not Borel complete via a faithful reduction. He showed that there is a strengthening of Vaught's conjecture---the Glimm-Effros dichotomy---which is true for linear orders but not true in general (and which is maintained by faithful Borel reductions).
	
	In Section \ref{sec:faithful}, we show that Theorem \ref{thm:SmallGaps} also implies that linear orders cannot be Borel complete via a faithful reduction. In general, we show that if we have a faithful Borel reduction from $\mathbb{K}$ to $\mathbb{L}$, then if $\mathbb{K}$ has Scott spectra gaps of unbounded sizes, then so does $\mathbb{L}$. This is, as far as we know, the only known reasonable method (i.e., other than ``brute force'') for proving that a Borel complete class of structures is not faithfully Borel complete without using some strengthening of Vaught's conjecture.
	
	Boolean algebras are the obvious class of structures in which to attempt to apply this. They are Borel complete, but the Borel reduction is not known to be faithful. In correspondence with Paolini, he has asked the following question:
	\begin{question}[Paolini]
		Are Boolean algebras faithfully Borel complete?
	\end{question}
	\noindent Vaught's conjecture is not known for Boolean algebras---despite many attempts---and so one might look to other methods to answer Paolini's question. We suggest that Scott spectral gaps are one possible method, and indeed the only method that we know of which would not first require proving Vaught's conjecture.
	
	Many of the key properties of linear orders that we use to prove Theorem \ref{thm:SmallGaps} correspond to analogous properties of Boolean algebras. However there is one fact we use about linear orders whose analog is not true for Boolean algebras: Lindenbaum and Tarski showed that if a linear order $M$ has initial segments of order type $L \cdot n$ for every $n$, then in fact $M \cong L + M$ and $M$ has an initial segment of order type $L \cdot \omega$ (see\footnote{Our understanding is that Tarski worked with Lindenbaum throughout the 20s and 30s to produce many results like this one on properties of linear orderings. Many of these results were announced without proof, e.g., in \cite{LindenbaumTarski}, or were only known informally for decades. They were not published in full until Tarski wrote this book after Lindenbaum's death. Despite the fact that we only cite a work of Tarski, we will associate Lindenbaum's name to this result as well.} Theorem 1.44 of \cite{TBook}). This is not true for Boolean algebras (and follows from \cite{Ket}). Thus, we know no bounds on the least Scott rank of the models of a theory of Boolean algebras.
	
	\begin{question}\label{question:Boolean}
		Given $\alpha$, does there exist $\beta$ such that if $T$ is a $\Pi_\alpha$ sentence expanding the theory of Boolean algebras, then $T$ has a model of Scott rank at most $\beta$? If so, what is the least such $\beta$?
	\end{question}
	
	\noindent In particular, we do not know the answer to the following question: Is there a bound $\beta$ depending on $\alpha$ such that if $\mc{A}$ and $\mc{B}$ are Boolean algebras, with $\mc{A}$ of Scott rank $\alpha$, and $\mc{A} + \mc{B} \equiv_\beta \mc{B}$, then $\mc{A} + \mc{B} \cong \mc{B}$. This is the analogue, for Boolean algebras, of the specific consequence we use of the result of Lindenbaum and Tarski described above.
	
	\medskip
	
	The proof of Theorem \ref{thm:SmallGaps} is a Henkin construction using formulas of bounded complexity, but crucially one must use a new complexity hierarchy of formulas which we call the $\A_\alpha$/$\E_\alpha$ formulas which are intimately tied to the back-and-forth analysis, but as far as we are aware, have not before been used. In particular, the classes satisfy $\Pi_\alpha \subsetneq \A_\alpha \subsetneq \forall_\alpha$ where the $\Pi_\alpha$ counts alternations of quantifiers including $\bigdoublewedge$ and $\bigdoublevee$ while $\forall_{\alpha}$ counts alternations of quantifiers but not counting $\bigdoublewedge$ and $\bigdoublevee$. These formulas are introduced in more detail in \cite{CGHT} which was written by the authors of this article together with Chen. These formulas seem necessary for the clear exposition of these results due to their presence in our key Lemma \ref{lem:intervalFormulas}, in which they are unavoidable as shown in Proposition \ref{prop:must-use-new}.
		
	We do not know if the bound $\alpha + 3$ on the Scott rank in the main theorem is optimal. Our best-known lower bound is $\alpha + 1$ (which, the reader should note, still gives a counterexample to Montalb\'an's Question \ref{question:BIRS}). For the case of $\lambda$ a limit ordinal, Gonzalez, Harrison-Trainor, and Ho \cite{GHTH} showed\footnote{In fact, this result was discovered while working on this paper, but uses the heavy machinery developed in that paper so it appears there instead.} that there is a $\Pi_\lambda$ sentence $T$ of linear orders all of whose models have Scott complexity $\Pi_{\lambda + 2}$ and hence Scott rank $\lambda+1$; we improve that here to include many successor ordinals as well as the limit case.
	
	\begin{theorem}
		For $\alpha = n \geq 4$, or $\alpha = \lambda$ a limit ordinal, or $\alpha = \lambda + n$ for $\lambda$ a limit ordinal and $n \geq 4$: There is a satisfiable $\Pi_{\alpha}$ sentence $T$ of linear orders such that every model of $T$ has no $\Sigma_{\alpha + 2}$ Scott sentence, and hence has Scott rank at least $\alpha + 1$.
	\end{theorem}
	
	\noindent Thus, the second part of Question \ref{question:SpectralGap}---of determining the optimal value of $\beta$---remains unsolved, but the gap between the lower and upper bounds has been reduced to $2$.
	
	What does one need to know to close this gap? There is a key question on which this seems to turn. If $\mc{A}$ has a $\Pi_\alpha$ Scott sentence, then if $\mc{A} \leq_\alpha \mc{B}$ then $\mc{A} \equiv_\alpha \mc{B}$; thus if $\mc{A}$ has a $\Pi_\alpha$ Scott sentence, then the $\alpha$-back-and-forth type of $\mc{A}$ is maximal. Is the converse true? That is, if the $\alpha$-back-and-forth type of $\mc{A}$ is maximal, must $\mc{A}$ have a $\Pi_\alpha$ Scott sentence?
	
	This question splits into two sub-questions, both of which were already known to the authors before this work but which we consider even more important now.
	
	\begin{question}
		Suppose that $\mc{A}$ has the property that whenever $\mc{B}$ is a countable structure and $\mc{A} \leq_\alpha \mc{B}$ then $\mc{B} \equiv_\alpha \mc{A}$. Then is it true that whenever $\mc{A} \leq_\alpha \mc{B}$ then $\mc{B} \cong \mc{A}$?
	\end{question}
	
	\begin{question}
		Suppose that $\mc{A}$ has the property that whenever $\mc{B}$ is a countable structure and $\mc{A} \leq_\alpha \mc{B}$ then $\mc{B} \cong \mc{A}$. Does $\mc{A}$ have a $\Pi_\alpha$ Scott sentence?
	\end{question}
	
	\noindent In \cite{CGHT} the authors and Chen show that in the second question $\mc{A}$ must have a $\Pi_{\alpha + 2}$ Scott sentence. Otherwise the best progress is a theorem of Montalb\'an \cite{MonSR} who showed that if there is a $\Pi_\alpha$ sentence $\varphi$ such that (a) $\mc{A} \models \varphi$ and (b) if $\mc{B} \models \varphi$ then $\mc{A} \equiv_\alpha \mc{B}$, then we can conclude that $\varphi$ is a $\Pi_\alpha$ Scott sentence for $\mc{A}$. 
	
	\medskip
	
	Finally, we note that though this theorem is about gaps between the complexity of the theory and the minimal element of the Scott spectrum, as a consequence we also get that all gaps are bounded. Suppose that $T$ is a $\Pi_\alpha$ sentence expanding the theory of linear orders. Let $\beta \geq \alpha$ and suppose that $T$ has a model $\mc{B}$ of Scott rank $\geq \beta$. Let $\varphi$ be the $\Pi_{2\beta}$ sentence which expresses of a model $\mc{A}$ that $\mc{A} \leq_{\beta} \mc{B}$. Then any $\mc{A} \models \varphi$ is a model of $T$, and has Scott rank $\geq \beta$. By Theorem \ref{thm:SmallGaps} $\varphi$ has a model $\mc{A}$ of Scott rank at most $2\beta+3$, and thus $\beta \leq SR(\mc{A}) \leq 2\beta + 3$. If, in notation to be introduced in the next section, we note that $\varphi$ is in fact $\overline{\E}_{\beta}$ hence $\E_{\beta+2}$ and use the form of the main theorem given in Theorem \ref{thm:main-inline} we get the following bound which is stronger except for at limit ordinals.
	
	\begin{corollary}
		Let $T$ be a $\Pi_\alpha$ sentence expanding the theory of linear orders, and let $\beta \geq \alpha$. If $T$ has a model of Scott rank $\geq \beta$, then it has a model $\mc{A}$ of Scott rank $\beta \leq SR(\mc{A}) \leq \beta + 5$.
	\end{corollary}
	\noindent This analysis transferring our main result to talk about the gaps in the Scott spectrum is quite coarse.
	It may be interesting to improve it in future work.
	
	The paper is organized into seven further sections. In Section \ref{EAComplexity}, we introduce the new notion $\A_\alpha$/$\E_\alpha$ of syntactic complexity needed for our main result. In Section \ref{LOBackground}, we establish background and results on linear orderings. In Section \ref{forcing}, we set up the machinery needed for the forcing construction in the main proof, which we prove in Section \ref{mainThm}. The next two sections construct example theories of linear orderings with non-trivial Scott spectrum gaps. In particular, in Section \ref{exampleLOs}, we construct base case examples, while in Section \ref{UpTheHeirarcy}, we develop a general mechanism for shifting examples in linear orderings to different levels of the hyperarithmetic hierarchy and apply these tools to our constructed examples. Finally, in Section \ref{sec:faithful} we have our results on faithful Borel reductions.
	
	\section{$\E_\alpha$ and $\A_\alpha$ formulas}\label{EAComplexity} 
	We introduce a new measure of syntactic complexity for infinitary formulas of $\Lomom$.
	We call our new complexity classes $\E_\alpha$ and $\A_\alpha$ for $\alpha\in\omega_1$.
	These syntactic classes are meant to line up well with the $\alpha$ back-and-forth relations in a manner that is analogous to the classical $\Sigma_\alpha$ and $\Pi_\alpha$ complexity classes.
	That said, the class of $\E_\alpha$ and $\A_\alpha$ formulas are far larger than $\Sigma_\alpha$ and $\Pi_\alpha$.
	Unlike these more traditional classes, our new classes do not line up with the hyperarithmetic hierarchy.
	Most importantly, the notion of $\E_\alpha$ and $\A_\alpha$ formulas is seemingly indispensable to the methods of this article, so we must introduce them here.

	These classes have many interesting properties that are independently interesting and that will be explored by the authors together with Chen in a distinct paper \cite{CGHT} with that focus.
	We will only presently cover the basic definitions and properties of these complexity classes that are needed for this paper.
	
	\begin{definition}
		All connectives below are countable.
		\begin{itemize}
			\item  $\A_1 := \Pi_1$
			\item $\E_1 := \Sigma_1$
			\item  $\A_\alpha :=$ closure of $\bigcup_{\beta < \alpha} \vwE_\beta$ under $\forall$ and $\bigwwedge$
			\item  $\E_\alpha :=$ closure of $\bigcup_{\beta < \alpha} \vwA_{\beta}$ under $\exists$ and $\bigvvee$
			\item  $\vwE_\alpha :=$ closure of $\E_\alpha$ under $\bigvvee, \bigwwedge$
			\item  $\vwA_\alpha :=$ closure of $\A_\alpha$ under $\bigvvee, \bigwwedge$
		\end{itemize}
		Each $\A_\alpha$ formula can be written in the form
		\[ \bigdoublewedge_{i \in I} \forall \bar{y}_i \varphi_i(\bar{x},\bar{y}_i)\]
		where the formulas $\varphi_i$ are $\vwE_\beta$ for some $\beta < \alpha$, and similarly a $\E_\alpha$ formula can be written in the form
		\[ \bigdoublevee_{i \in I} \exists \bar{y}_i \varphi_i(\bar{x},\bar{y}_i)\]
		where each $\varphi_i$ is $\vwA_\beta$ for some $\beta < \alpha$.
	\end{definition}
	
	In words, the classes $\E_\alpha$ and $\A_\alpha$ are similar to $\Sigma_\alpha$ and $\Pi_\alpha$ except for the fact that before adding a new alternation of quantifiers, they are allowed to insert an arbitrary alternation of countable conjunctions and disjunctions.
	We enumerate some of the basic properties of $\E_\alpha$ and $\A_\alpha$ below.
	We omit the proofs of these properties, as they are all straightforward transfinite inductions, or they follow immediately from the definitions.

	\begin{remark}\label{rmk:containments}
		Up to equivalence in countable structures:
		\begin{itemize}
			\item $\lnot \E_\alpha = \A_\alpha$,  $\lnot \vwE_\alpha = \vwA_\alpha$
			\item  $\Sigma_\alpha \subseteq \E_\alpha \subseteq\vwE_\alpha$
			\item
			$\bigvvee \E_\alpha \subseteq \E_\alpha$,
			$\bigwwedge \E_\alpha \subseteq \vwE_\alpha$,
			$\exists \E_\alpha \subseteq \E_\alpha$,
			$\forall \E_\alpha \subseteq \A_{\alpha+1}$
			\item
			$\bigvvee \vwE_\alpha \subseteq \vwE_\alpha$,
			$\bigwwedge \vwE_\alpha \subseteq \vwE_\alpha$,
			$\exists \vwE_\alpha \subseteq \E_{\alpha+2}$,
			$\forall \vwE_\alpha \subseteq \A_{\alpha+1}$
			\item
			for limit $\alpha$: $\vwE_\alpha = \vwA_\alpha =$ closure of $\bigcup_{\beta < \alpha} (\vwE_\beta \cup \vwA_\beta)$ under $\bigvvee, \bigwwedge$
		\end{itemize}
	\end{remark}

As indicated above, the key properties of this hierarchy lie in the way that it interacts with the $\alpha$ back-and-forth relations.
When we refer to the $\alpha$ back-and-forth relations, we mean the \textit{standard asymmetric back-and-forth relations} defined as follows.

	\begin{definition}\label{def:bfasym}
	  $\leq_\alpha$, for $\alpha < \omega_1$, are defined by:
		\begin{itemize}
			\item $(\mc{M},\bar{a}) \leq_0 (\mc{N},\bar{b})$ if $\bar{a}$ and $\bar{b}$ satisfy the same quantifier-free formulas from among the first $|\bar{a}|$-many formulas.
			\item For $\alpha > 0$, $(\mc{M},\bar{a}) \leq_\alpha (\mc{N},\bar{b})$ if for each $\beta < \alpha$ and $\bar{d} \in \mc{N}$ there is $\bar{c} \in \mc{M}$ such that $(\mc{N},\bar{b} \bar{d}) \leq_\beta (\mc{M},\bar{a} \bar{c})$.
		\end{itemize}
		We define $\bar{a} \equiv_\alpha \bar{b}$ if $\bar{a} \leq_\alpha \bar{b}$ and $\bar{b} \leq_\alpha \bar{a}$.
	\end{definition}
	\noindent The interpretation of $(\mc{M},\bar{a}) \leq_\alpha (\mc{N},\bar{b})$ is that in the back-and-forth game between $\mc{M}$ and $\mc{N}$, starting with the partial isomorphism $\bar{a} \mapsto \bar{b}$ and with the first player \textsf{Spoiler} to play next in $\mc{N}$, the second player \textsf{Duplicator} can play without losing along an ordinal clock $\alpha$. Recall that if \textsf{Duplicator} can continue playing forever, then $\mc{M} \cong \mc{N}$.

We show first that the $\alpha$ back-and-forth relations guarantee agreement on $\vwE_\alpha$ and $\vwA_\alpha$ formulas the same way that they guarantee agreement on $\Sigma_\alpha$ and $\Pi_\alpha$ formulas, even though there are more $\vwE_\alpha$ and $\vwA_\alpha$ formulas.
In particular, we aim to extend the following theorem of Karp (see e.g. \cite{MBook}, Theorem II.36).

\begin{theorem}[Karp \cite{Karp}]\label{thm:Karp}
	For any $\alpha \geq 1$, structures $\mc{A}$ and $\mc{B}$, and tuples $\bar{a}\in\mc{A}$ and $\bar{b}\in\mc{B}$, the following are equivalent:
	\begin{enumerate}
		\item $(\mc{A},\bar{a})\leq_\alpha (\mc{B},\bar{b})$.
		\item Every $\Pi_\alpha$ formula true about $\bar{a}$ in $\mc{A}$ is true about $\bar{b}$ in $\mc{B}$.
		\item Every $\Sigma_\alpha$ formula true about $\bar{b}$ in $\mc{B}$ is true about $\bar{a}$ in $\mc{A}$.
	\end{enumerate} 
\end{theorem} 
	
\begin{lemma}[Chen, Gonzalez, and Harrison-Trainor \cite{CGHT}]\label{prop:transfer-over-bf}
	Suppose that $(\mc{A},\bar{a}) \leq_\alpha (\mc{B},\bar{b})$ for countable structures $\mc{A}$, $\mc{B}$, and $\alpha \geq 1$. Then given a $\vwE_\alpha$ formula $\varphi(\bar{x})$ and a $\vwA_\alpha$ formula $\psi(\bar{x})$,
	\[ \mc{B} \models \varphi(\bar{b}) \Rightarrow \mc{A} \models \varphi(\bar{a})\]
	and
	\[ \mc{A} \models \psi(\bar{a}) \Rightarrow \mc{B} \models \psi(\bar{b})\]
\end{lemma}
	\begin{proof}[Proof idea]
				 Argue by induction on complexity. 
			\if{false}
				Our base case is $\alpha = 1$ and $\A_1$ and $\E_1$ formulas, which are just $\Pi_1$ and $\Sigma_1$ formulas. For these, the proposition follows directly from Theorem \ref{thm:Karp}.	Now we have the inductive cases. In the inductive argument, we use the fact that we can group the outside quantifiers in a $\A_\alpha$ or $\E_\alpha$ sentence, as described above.
		
		Suppose that $\varphi(\bar{x})$ is $\E_\alpha$, say $\varphi(\bar{x}) = \bigdoublevee_i \exists \bar{y}_i \theta_i(\bar{x},\bar{y}_i)$ where the $\theta_i(\bar{x},\bar{y}_i)$ are $\vwA_\beta$ for some $\beta < \alpha$. Suppose that $\mc{B} \models \varphi(\bar{b})$; then there is $i$ and $\bar{b}'$ such that $\mc{B} \models \theta_i(\bar{b},\bar{b}')$. Then there is $\bar{a}'$ such that $(\mc{B},\bar{b}\bar{b}') \leq_\beta (\mc{A},\bar{a}\bar{a}')$. By the induction hypothesis, $\mc{A} \models \theta_i(\bar{a},\bar{a}')$. Thus $\mc{A} \models \varphi(\bar{a})$ as desired.
		
		Suppose that $\varphi(\bar{x})$ is $\A_\alpha$, say $\varphi(\bar{x}) = \bigdoublewedge_i \forall \bar{y}_i \theta_i(\bar{x},\bar{y}_i)$ where the $\theta_i(\bar{x},\bar{y}_i)$ are $\vwE_\beta$ for some $\beta < \alpha$. Suppose that $\mc{A} \models \varphi(\bar{a})$; then for every $\bar{a}'\in\mc{A}$ and $i$,  $\mc{A} \models \theta_i(\bar{a},\bar{a}')$.  Suppose there is some $\bar{b'}\in\mc{B}$ and $i$ such that $\mc{B}\not\models \theta_i(\bar{b},\bar{b}')$. Then there is $\bar{a}'$ such that $(\mc{B},\bar{b}\bar{b}') \leq_\beta (\mc{A},\bar{a}\bar{a}')$. By the induction hypothesis, $\mc{A} \not\models \theta_i(\bar{a},\bar{a}')$, a contradiction. Thus $\mc{B} \models \varphi(\bar{b})$ as desired.
		
		Suppose that $\varphi(\bar{x})$ is $\vwE_\alpha$, and say $\varphi(\bar{x}) = \bigdoublevee_i \theta_i(\bar{x})$ where each $\theta_i(\bar{x})$ is $\vwE_\alpha$. Suppose that $\mc{B} \models \varphi(\bar{b})$; then for some $i$, $\mc{B} \models \theta_i(\bar{b})$, and so by the induction hypothesis, $\mc{A} \models \theta_i(\bar{a})$. Thus $\mc{A} \models \varphi(\bar{a})$. Similarly, suppose that $\varphi(\bar{x})$ is $\vwE_\alpha$, and say $\varphi(\bar{x}) = \bigdoublewedge_i \theta_i(\bar{x})$ where each $\theta_i(\bar{x})$ is $\vwE_\alpha$. Suppose that $\mc{B} \models \varphi(\bar{b})$; then for all $i$, $\mc{B} \models \theta_i(\bar{b})$, and so by the induction hypothesis, $\mc{A} \models \theta_i(\bar{a})$ for all $i$. Thus $\mc{A} \models \varphi(\bar{a})$.
		
		The cases where $\varphi(\bar{x})$ is $\vwA_\alpha$, and either of the form $\varphi(\bar{x}) = \bigdoublevee_i \theta_i(\bar{x})$ or $\varphi(\bar{x}) = \bigdoublewedge_i \theta_i(\bar{x})$ where each $\theta_i(\bar{x})$ is $\vwA_\alpha$, is identical to the previous case.
		\fi
		\end{proof}

The following is immediate from the above lemma, along with Karp's theorem.
		
\begin{corollary}[Chen, Gonzalez, and Harrison-Trainor \cite{CGHT}]\label{cor:karp}
	For any $\alpha \geq 1$, structures $\mc{A}$ and $\mc{B}$, and tuples $\bar{a}\in\mc{A}$ and $\bar{b}\in\mc{B}$, the following are equivalent:
	\begin{enumerate}
		\item $(\mc{A},\bar{a})\leq_\alpha (\mc{B},\bar{b})$.
		\item Every $\Pi_\alpha$ formula true about $\bar{a}$ in $\mc{A}$ is true about $\bar{b}$ in $\mc{B}$.
		\item Every $\Sigma_\alpha$ formula true about $\bar{b}$ in $\mc{B}$ is true about $\bar{a}$ in $\mc{A}$.
		\item Every $\A_\alpha$ formula true about $\bar{a}$ in $\mc{A}$ is true about $\bar{b}$ in $\mc{B}$.
		\item Every $\E_\alpha$ formula true about $\bar{b}$ in $\mc{B}$ is true about $\bar{a}$ in $\mc{A}$.
		\item Every $\vwA_\alpha$ formula true about $\bar{a}$ in $\mc{A}$ is true about $\bar{b}$ in $\mc{B}$.
		\item Every $\vwE_\alpha$ formula true about $\bar{b}$ in $\mc{B}$ is true about $\bar{a}$ in $\mc{A}$.
	\end{enumerate} 
\end{corollary}

This means that when it comes to interactions with back-and-forth relations, our new notions of complexity act at least as nicely as the classical notions.
That said, our new notions have properties that the classical ones do not.
The key difference is distilled in the following theorem.
Classically, there may be no maximal  $\Pi_\alpha$ or $\Sigma_\alpha$ formula, i.e. one which captures the entire $\Pi_\alpha$ or $\Sigma_\alpha$ theory of a given structure or tuple.
In particular, one may need a $\Pi_{2\alpha}$ formula to describe theories at this level (see \cite{MBook} Lemma VI.14).
Our larger classes contrast with this; they are always able to capture an entire theory at level $\alpha$ with a formula at that same level.
	
	\begin{theorem}[Chen, Gonzalez, and Harrison-Trainor \cite{CGHT}]\label{thm:bnfFormulas}
		For $\mc{A}$ a countable structure, $\bar{a} \in \mc{A}$, and $\alpha \geq 1$, there are $\vwE_\alpha$ formulas $\varphi_{\bar{a},\mc{A},\alpha}(\bar{x})$ and $\A_\alpha$ formulas $\psi_{\bar{a},\mc{A},\alpha}(\bar{x})$ such that for $\mc{B}$ any structure,
		\[ \mc{B} \models \varphi_{\bar{a},\mc{A},\alpha}(\bar{b}) \Longleftrightarrow (\mc{B},\bar{b}) \leq_\alpha (\mc{A},\bar{a})\]
		and
		\[ \mc{B} \models \psi_{\bar{a},\mc{A},\alpha}(\bar{b}) \Longleftrightarrow (\mc{B},\bar{b}) \geq_\alpha (\mc{A},\bar{a}).\]
	\end{theorem}
	\begin{proof}[Proof idea]
		We argue by induction, with the key case being $\alpha = 1$. For $\psi_{\bar{a},\mc{M},1}$, for each $m$ there are only finitely many possible atomic $m$-types in $\mc{M}$, only finitely many of which are realized in $\mc{M}$. Let $\theta_{\bar{a}}(\bar{x})$ be the finitary quantifier-free formula which says that $\bar{x}$ and $\bar{a}$ satisfy the same atomic formulas (from among the first $|\bar{a}|$-many formulas; in a finite relational language, like that of linear orders, we can take $\theta_{\bar{a}}(\bar{x})$ to say that $\bar{a}$ and $\bar{x}$ have the same atomic type, i.e., are ordered in the same way.) Then $(\mc{N},\bar{b}) \geq_1 (\mc{M},\bar{a})$ if and only if
		\[ \mc{N} \models \bigdoublewedge_{n \in \mathbb{N}} \forall \bar{y}^n \bigdoublevee_{\bar{a}' \in \mc{M}^n} \theta_{\bar{a}\bar{a}'}(\bar{b},\bar{y}).\]
		The inner disjunction looks like it is infinite, but it is not; this is because there are only finitely many possible formulas $\theta_{\bar{a}\bar{a}'}(\bar{x})$. Thus, this formula is $\A_1$. The remainder of the induction is straightforward following \cite{MBook} Lemma VI.14.
		\if{false}
		For $\varphi_{\bar{a},\mc{M},1}$, we have $(\mc{N},\bar{b}) \leq_1 (\mc{M},\bar{a})$ if and only if
		\[ \mc{N} \models \bigdoublewedge_{\bar{a}' \in \mc{M}} \exists \bar{y}  \theta_{\bar{a}\bar{a}'}(\bar{b},\bar{y}).\]
		This is $\vwE_1$.

		Now, given $\alpha > 1$, suppose that we have $\vwE_\beta$ formulas $\varphi_{\bar{a},\mc{M},\beta}$ and $\A_\beta$ $\psi_{\bar{a},\mc{M},\beta}$ for $\beta < \alpha$. Then $(\mc{N},\bar{b}) \leq_{\alpha} (\mc{M},\bar{a})$ if and only if
		\[ \bigdoublewedge_{\beta < \alpha} \bigdoublewedge_{\bar{a}' \in \mc{M}} \exists \bar{y} \psi_{\bar{a}\bar{a}',\mc{M},\beta}(\bar{b},\bar{y}) \]
		and
		$(\mc{N},\bar{b}) \geq_{\alpha} (\mc{M},\bar{a})$ if and only if
		\[ \bigdoublewedge_{\beta < \alpha} \bigdoublewedge_{m \in \mathbb{N}} \forall \bar{y}^m \bigdoublevee_{\bar{a}' \in \mc{M}} \varphi_{\bar{a}\bar{a}',\mc{M},\beta}(\bar{b},\bar{y}).\]
		These define our desired $\vwE_\alpha$ formulas $\varphi_{\bar{a},\mc{M},\alpha}$ and $\A_\alpha$ formulas $\psi_{\bar{a},\mc{M},\alpha}$.
		\fi
	\end{proof}

	\section{Linear orders and splitting into intervals}\label{LOBackground}
	
	Our aim in this section is to prove an important technical lemma about linear orderings.
	This lemma describes how to take a formula that is about a tuple of elements in a linear ordering and find a stronger formula that only refers to particular intervals inside of the linear ordering.
	In essence, the lemma helps us eliminate parameters from formulas about linear orderings in a manner that is particularly expedient for the proof of the main theorem.
	
	We begin by introducing some notation. Given a linear order $\mc{A}$ and a tuple $\bar{a} = (a_1,\ldots,a_n) \in \mc{A}$, listed in increasing order, we will frequently split $\mc{A}$ up into the intervals $(-\infty,a_1),(a_1,a_2),\ldots,(a_n,\infty)$.	Throughout the paper, will make use of this division into intervals, and so we will adopt the convention that we use $a_0$ for $-\infty$ and $a_{n+1}$ for $\infty$; thus, the intervals $(a_i,a_{i+1})$ for $i = 0,\ldots,n$ are the intervals into which $\bar{a}$ splits $\mc{A}$. There will be a few places where we define notation to take advantage of this so that we do not need to split into cases based on whether or not an interval has a bound at infinity.
	
	Given an interval $(a,b)$ in a linear order $\mc{A}$, $(a,b)$ is a linear order in its own right, and hence we can write $(a,b) \models \varphi$ for a sentence $\varphi$. The fact that $\varphi$ is true in $(a,b)$ can also be witnessed in $\mc{A}$ by relativizing the formula $\varphi$ to the interval $(a,b)$. Let $\varphi \Res_{(x,y)}$ be the formula in two variables, which relativizes all of the quantifiers of $\varphi$ to the interval $(x,y)$. Then
	\[ \mc{A} \models \varphi \Res_{(a,b)} \;\; \Longleftrightarrow (a,b) \models \varphi.\]
	The quantifier complexity of the relativized formula will be the same as that of the original formula. We can also relativize $\varphi$ to intervals $(-\infty,x)$ and $(y,\infty)$. This will be useful using the convention above that given elements $a_1 < \ldots < a_n$, we let $a_0$ represent $-\infty$ and $a_{n+1}$ represent $\infty$, so that, e.g., given sentences $\theta_0,\ldots,\theta_n$,
	\[ \mc{A} \models \bigwedge_{i=0}^n \theta_i \Res_{(a_i,a_{i+1})} \]
	expresses that after dividing $\mc{A}$ into $n+1$ intervals $(-\infty,a_1),(a_1,a_2),\ldots,(a_n,\infty)$, each interval satisfies the corresponding sentence $\theta_i$. Note that $\varphi \Res_{(-\infty,x)}$ and $\varphi \Res_{(y,\infty)}$ are formulas with only one free variable. Sometimes it will make more sense to write $\varphi \Res_{< x}$ for $\varphi \Res_{(-\infty,x)}$ and $\varphi \Res_{>y}$ for $\varphi \Res_{(y,\infty)}$. We will particularly use this latter notation when working within an interval in a larger linear order, e.g., given an interval $(a,b) \in \mc{A}$, and $c \in (a,b)$, we will tend to write $(a,b) \models \varphi \Res_{< c}$ rather than $(a,b) \models \varphi \Res_{(-\infty,c)}$ though they mean the same thing.
	
	It is well-known that, given elements $a_1 < a_2 < \cdots < a_n$ in a linear order, the back-and-forth type of the tuple $\bar{a}$ is determined by the back-and-forth types of the intervals between the elements, namely the intervals $(a_i,a_{i+1})$ for $i = 0,1,\ldots,n$ (recalling our notation above that these are the intervals $(-\infty,a_1),(a_1,a_2),\dots,(a_n,\infty)$). More formally:
	
	\begin{lemma}[\cite{AK00} Lemma 15.8]\label{lem:combine-bf}
		Given linear orders $(\mc{A},\bar{a})$ and $(\mc{B},\bar{b})$ (with the tuples in increasing order), $(\mc{A},\bar{a}) \geq_\alpha (\mc{B},\bar{b})$ if and only if for each $i = 0,\ldots,n$ we have $(a_i,a_{i+1})^{\mc{A}} \geq_\alpha (b_i,b_{i+1})^{\mc{B}}$.
	\end{lemma}
	
	We would like a version of this for individual sentences. This is a more precise intuition for what our lemma achieves. We had hoped that the following statement is true, but we now know it is false: Given a $\Sigma_\alpha$ formula $\varphi(x_1,\ldots,x_n)$, there are $\Sigma_\alpha$ sentences $\theta^j_0,\ldots,\theta^j_n$ ($j \in J$) such that, for all linear orders $\mc{A}$ and $a_1<\cdots<a_n$, $\mc{A} \models \varphi(a_1,\ldots,a_n)$ if and only if, for some $j \in J$, and each $i = 0,\ldots,n$, $(a_i,a_{i+1}) \models \theta^j_i$. In fact, there seem to be two issues. First, if one tries to prove this as written by induction, one gets completely stuck at the case of a universal quantifier. This leads one to give up on a two-way implication and instead find a sufficient condition on the intervals $(a_i,a_{i+1})$---though non-vacuous, i.e., satisfiable in some model---to have $\varphi$ hold. Even in proving this, one finds that one cannot keep the same complexity as the original formulas if one measures complexity using the $\Sigma_\alpha$/$\Pi_\alpha$ hierarchy. Instead, one is naturally led to use the $\E_\alpha$/$\A_\alpha$ hierarchy. (In Section \ref{exampleLOs}, we give an explicit example showing that one cannot remain within the $\Sigma_\alpha$/$\Pi_\alpha$ hierarchy.)
	
	\begin{lemma}\label{lem:intervalFormulas}
		Let $\mc{L}$ be a countable linear order and $a_1<\cdots<a_n$ elements of $\mc{L}$. Suppose that $\mc{L}\models\varphi(a_1,\ldots,a_n)$ with $\varphi$ a $\E_\alpha$ formula in the language of linear orders. Then there are $\E_\alpha$ sentences $\theta_0,\ldots,\theta_n$ such that (a) for every $k = 0,\ldots,n$ we have $(a_k,a_{k+1}) \models \theta_k$, and (b) if $\mc{B}$ is any linear order and $b_1 < \cdots < b_n$, if for every $k = 0,\ldots,n$ we have $(b_k,b_{k+1}) \models \theta_k$ then $\mc{B} \models \varphi(b_1,\ldots,b_n)$. 
	\end{lemma}
	
	Using the notation established above, we can also write (b) in the conclusion as saying that
	\[ \models \left[ (x_1 < x_2 < \cdots < x_n) \wedge \bigwedge_{i=0}^n \theta_i \Res_{(x_i,x_{i+1})} \right] \longrightarrow \varphi(x_1,\ldots,x_n). \]
	The purpose of (a) is to say that the antecedent is consistent and, indeed satisfied by $\bar{a}$ in $\mc{A}$, so that (b) is not vacuous.
	

	\begin{proof}
		Write $\varphi(\bar{x})$ as $\bigvvee_k \exists \bar{y} ~ \psi_k(\bar{a},\bar{y})$ with $\psi_k\in\vwA_\beta$ for $\beta<\alpha$.
		As $\mc{L}\models \varphi(\bar{a})$, there is an index $k$ and witness $\bar{c}$ for which $\mc{L}\models \psi_k(\bar{a},\bar{c})$.
		Say that $\bar{a},\bar{c}$ can be written in order with the following indexing:
		\[ c_{0,1} < \cdots < c_{0,m_0} < a_1 < c_{1,1} < \cdots < c_{1,m_1} < a_2 < \cdots < a_3 < \cdots < a_n < c_{n,1} < \cdots < c_{n,m_n}. \]
		By convention, let $c_{0,0}=-\infty$, $c_{i,0}=a_{i-1}$ for $i>0$,  $c_{i,m_i+1}=a_{i}$ for $i<n$ and $c_{n,m_n+1}=\infty$.
		Using Theorem \ref{thm:bnfFormulas}, for each interval represented above, there is a $\A_\beta$ formula such that $\mc{N}\models\chi_{i,j}$ if and only if $\mc{N}\geq_\beta (c_{i,j},c_{i,j+1})$.
		For each $i = 0,\ldots,n$ we can then define an $\E_\alpha$ formulas as follows:
		\[\theta_i := \exists x_1,\cdots, x_{m_i} \bigwedge_{j\leq m_j} \chi_{i,j}\Res (x_j,x_{j+1}).\]
		We claim these formulas satisfy properties (a) and (b) from the statement of the lemma.
		First, it is clear that $c_{i,1} < \cdots < c_{i,m_i}$ are appropriate witnesses in each $(a_i,a_{i+1})$, and therefore (a) holds.
		Next, we show (b).
		Say that $(\mc{B},\bar{b})$ has the property that each $(b_i,b_{i+1})\models \theta_i$.
		We can write out the witnesses $\bar{d}$ for each $\theta_i$ in order with the following indexing:
		\[ d_{0,1} < \cdots < d_{0,m_0} < b_1 < d_{1,1} < \cdots < d_{1,m_1} < b_2 < \cdots < b_3 < \cdots < b_n < d_{n,1} < \cdots < d_{n,m_n}, \]
		and similarly take the convention that $d_{0,0}=-\infty$, $d_{i,0}=b_{i-1}$ for $i>0$,  $d_{i,m_i+1}=b_{i}$ for $i<n$ and $d_{n,m_n+1}=\infty$.
		Because $(d_{i,j},d_{i,j+1})\models \chi_{i,j}$, we have that $(d_{i,j},d_{i,j+1})\geq_\beta (c_{i,j},c_{i,j+1})$.
		In particular, by Lemma \ref{lem:combine-bf}, we have that $(\bar{b},\bar{d})\geq_\beta (\bar{a},\bar{c})$.
		This means that $\mc{B}\models \psi_k(\bar{b},\bar{d})$ and therefore, $\mc{B}\models \bigvvee_k\exists\bar{y} ~\psi_k(\bar{b},\bar{y})$ and $\mc{B}\models \varphi(\bar{b})$, as desired.
	\end{proof}
	
	The example demonstrating the necessity of using the $\E_\alpha$ and $\A_\alpha$ hierarchies for the above result is as follows.
	
	\begin{restatable}{proposition}{usenew}\label{prop:must-use-new}
		There is a $\Pi_4$ sentence $\theta$ expanding the theory of linear orders such that for any consistent $\Pi_4$ sentences $\varphi$ and $\psi$ there are $\mc{A} \models \varphi$ and $\mc{B} \models \psi$ such that $\mc{A} + 1 + \mc{B} \nmodels \theta$.
	\end{restatable}

	To prove this, it is easier to work in the language with additional unary relations $\{ R_i : i \in \omega\}$. One can then eliminate these relations at the cost of adding a few quantifiers by using the techniques of Section \ref{exampleLOs} to obtain the above proposition. We delay the proof of Proposition \ref{prop:must-use-new} to Section \ref{exampleLOs} where we will have established the described techniques precisely.

	\section{The Forcing Machinery}\label{forcing}
	
	In this section, we set up the forcing machinery needed to construct the models of appropriate Scott ranks.
	In particular, we can think of our model construction as a sort of model-theoretic forcing and our constructed model as a generic element with respect to this forcing.
	We aim to make this intuition precise in this section.
	
	We start by making the following definition of an existential fragment.
	
	\begin{definition}
		Let $\tau$ be a signature. A set of $\mc{L}_{\omega_1 \omega}$ formulas $\mathbb{A}$ is an existential fragment if there is an infinite set of variables $V$ such that all formulas in $\mathbb{A}$ have all variables coming from $V$, and $\mathbb{A}$ satisfies the following closure properties:
		\begin{enumerate}
			\item All atomic and negated atomic formulas using only the constant symbols of $\tau$ and variables from $V$ are in $\mathbb{A}$.
			\item If $\varphi \in \mathbb{A}$ and $\psi$ is a subformula of $\varphi$ then $\psi \in \mathbb{A}$.
			\item If $\varphi \in \mathbb{A}$, $v$ is free in $\varphi$, and $t$ is a term where every variable is in $V$, then the formula obtained by substituting $t$ into all free occurrences of $v$ is in $\mathbb{A}$.
			\item $\mathbb{A}$ is closed under finite conjunctions and disjunctions $\wedge$ and $\vee$ and under existential quantification $\exists v$ for $v \in V$.
		\end{enumerate}
	\end{definition}
	
	\noindent For clarity, we will use the term full fragment for the traditional notion of a fragment. The difference between the two is that a full fragment is closed under formal negation $\sim$ and universal quantification $\forall v$.
	
	While every formula is contained in a countable full fragment, the formulas in the fragment may have greater quantifier complexity than the original formula. Existential fragments are designed specifically to maintain the same level of complexity. Given a $\E_\alpha$ formula $\varphi$, there is a countable existential fragment $\mathbb{A}$ containing $\varphi$ such that every formula of $\mathbb{A}$ is $\E_\alpha$. Moreover, we also want the fragment to be ``closed under Lemma \ref{lem:intervalFormulas}'' in the following sense.
	
	\begin{lemma}\label{lem:CountableSet}
		Let $\tau$ be the signature of linear orders together with countably many constant symbols. Given a $\E_\alpha$ formula $\varphi$ with $\mc{L}\models\varphi$, there is a countable existential fragment $\mathbb{A}$ containing $\varphi$, such that every formula of $\mathbb{A}$ is $\E_\alpha$, and a countable class $\mathbb{C}$ of countable linear orders such that:
		\begin{enumerate}
			\item whenever $\varphi \in \mathbb{A}$, and $x,y \in V$, $\varphi \Res_{(x,y)}$, $\varphi \Res_{< x}$, and $\varphi \Res_{> y}$ are in $\mathbb{A}$;
			\item whenever $\varphi(\bar{x}) \in \mathbb{A}$ is consistent there is $\mc{A} \in \mathbb{C}$ and $\bar{a} \in \mc{A}$ such that $\mc{A} \models \varphi(\bar{a})$;
			\item if $\mc{A} \in \mathbb{C}$ and $\varphi(\bar{x}) \in \mathbb{A}$ and $\mc{A} \models \varphi(\bar{a})$ with $a_1 < \cdots < a_n$ then there are sentences $\theta_0,\ldots,\theta_n \in \mathbb{A}$ satisfying the conclusion of Lemma \ref{lem:intervalFormulas}, i.e., such that (a) for every $k = 0,\ldots,n$ we have $(a_k,a_{k+1}) \models \theta_k$, and (b) if $\mc{B}$ is any linear order and $b_1 < \cdots < b_n$, if for every $k = 0,\ldots,n$ we have $(b_k,b_{k+1}) \models \theta_k$ then $\mc{B} \models \varphi(b_1,\ldots,b_n)$.
		\end{enumerate}
	\end{lemma}	
	\begin{proof}
		We construct $\mathbb{A}$ and $\mathbb{C}$ by stages, $\mathbb{A} = \bigcup \mathbb{A}_s$ and $\mathbb{C} = \bigcup \mathbb{C}_s$. Let $V$ be a countable set of variables including those in $\varphi$. We begin with $\mathbb{A}_0$ consisting of $\varphi$ together with all atomic and negated atomic formulas using only the constant symbols of $\tau$ and the variables from $V$, and $\mathbb{C}_0 = \{\mc{L}\}$. At each stage, we close under one of the closure conditions, either one of the three closure conditions to be an existential fragment, or (1), (2), or (3). At each stage, $\mathbb{A}_s$ and $\mathbb{C}_s$ will be countable, and we add only countably many new formulas or structures at each stage so that $\mathbb{A}$ and $\mathbb{C}$ are countable unions of countable sets and hence countable. For the conditions to be an existential fragment, it is clear how to take one step towards closure, e.g., to take a step towards being closed under subformulas, given $\mathbb{A}_s$, we can take $\mathbb{A}_{s+1}$ to consist of $\mathbb{A}_s$ together with all subformulas of formulas in $\mathbb{A}_s$, and $\mathbb{C}_{s+1} = \mathbb{C}_s$. It is also easy to see how to take a step towards (1). It is important to note that in each of these cases, the new formulas we add are all $\E_\alpha$.
		
		To close under (2), given $\mathbb{A}_s$ and $\mathbb{C}_s$, let $\mathbb{A}_{s+1} = \mathbb{A}_s$. For each $\varphi(\bar{x}) \in \mathbb{A}_s$ which is consistent, choose a countable structure $\mc{A}$ and $\bar{a} \in \mc{A}$ such that $\mc{A} \models \varphi(\bar{a})$. Let $\mathbb{C}_{s+1}$ be $\mathbb{C}_s$ together with these structures $\mc{A}$ for all choices of $\varphi(\bar{x}) \in \mathbb{A}_s$. (Note that we use downwards Lowenheim-Skolem for $\mc{L}_{\omega_1 \omega}$ here to get countable structures.)
		
		To close under (3), given $\mathbb{A}_s$ and $\mathbb{C}_s$, let $\mathbb{C}_{s+1} = \mathbb{C}_s$. For each $\mc{A} \in \mathbb{C}_s$ and $\varphi(\bar{x}) \in \mathbb{A}_s$ and $a_1 < \cdots < a_n$ in $\mc{A}$ such that $\mc{A} \models \varphi(\bar{a})$, by Lemma \ref{lem:intervalFormulas} choose sentences $\theta_0,\ldots,\theta_n$ $\varphi(\bar{x}) \in \mathbb{A}_s$ such that (a) for every $k = 0,\ldots,n$ we have $(a_k,a_{k+1}) \models \theta_k$, and (b) if $\mc{B}$ is any linear order and $b_1 < \cdots < b_n$, if for every $k = 0,\ldots,n$ we have $(b_k,b_{k+1}) \models \theta_k$ then $\mc{B} \models \varphi(b_1,\ldots,b_n)$. Let $\mathbb{A}_{s+1}$ be $\mathbb{A}_s$ together with all of these sentences $\theta_0,\ldots,\theta_n$ for all choices of $\mc{A}$, $\varphi(\bar{x})$, and $\bar{a}$.
	\end{proof}
	
	Fix for now a countable existential fragment $\mathbb{A}$ consisting of $\E_\alpha$ formulas. Let $\mathbb{A}^*$ be the satisfiable sentences of $\mathbb{A}$; for the remainder of this section, all sentences are assumed to be in $\mathbb{A}^*$. We define a forcing relation on $\mathbb{A}^*$. Put $\varphi \leq \psi$ ($\varphi$ is an extension of $\psi$) if $\varphi \models \psi$. Note that if $\varphi$ and $\psi$ are consistent, then $\varphi \wedge \psi \in \mathbb{A}$ is a condition with $\varphi \wedge \psi \leq \varphi$ and $\varphi \wedge \psi \leq \psi$. (Thus, we might equivalently force with finite consistent sets of sentences from $\mathbb{A}$.)
	
	We say that $\psi_1 \leq \varphi$ and $\psi_2 \leq \varphi$ \textit{split over} $\varphi$ if $\psi_1 \wedge \psi_2$ is not consistent; equivalently, there is no common extension of $\psi_1$ and $\psi_2$. We say that $\varphi$ \textit{forces unity} if there are no splits over $\varphi$; thus, given any $\psi_1 \leq \varphi$ and $\psi_2 \leq \varphi$, there is a common extension $\psi_1 \wedge \psi_2$, and $\psi_1 \wedge \psi_2 \wedge \varphi$ is consistent. We say that $\varphi$ \textit{forces splitting} if it forces the negation of unity, that is, for any $\varphi' \leq \varphi$, there are $\psi_1 \leq \varphi'$ and $\psi_2 \leq \varphi'$ such that $\psi_1$ and $\psi_2$ split over $\varphi'$.
	
	As usual, any sentence $\varphi$ has an extension that either forces unity or forces splitting.
	
	\begin{lemma}
		Given $\varphi \in \mathbb{A}^*$, there is either a $\psi \leq \varphi$ that forces unity or a $\psi \leq \varphi$ that forces splitting.
	\end{lemma}
	
	
	\section{Proof of the Main Theorem}\label{mainThm}
	
	We are now ready to prove the main result of the paper.
	Our proof has two stages: the construction and verification.
	To aid the exposition, we isolated two particular properties of linear orderings: being generic relative to an existential fragment $\mathbb{A}^*$ and the so-called Property $\mathbf{(*)}$.
	In the construction stage, we show that we can create linear orderings that are generic relative to an existential fragment $\mathbb{A}^*$, have Property $\mathbf{(*)}$, and satisfy a given $\E_\alpha$ formula.
	In the verification stage, we show that linear orderings with these properties also have small Scott ranks.
	We present our work in the following order, which we feel to be the most clear: first, we formally define the properties of linear orderings that we are dealing with, then we demonstrate the verification stage, and finally, we demonstrate the construction stage.
	
	We will make use of the following classical result about linear orders. Together with Lemma \ref{lem:intervalFormulas}, these are the only facts about linear orders that we use.

	\begin{proposition}[Lindenbaum and Tarski, see Theorem 1.44 of \cite{TBook}]\label{Tarski}
		If $N$ and $L$ are order types, and $N \cdot k$ is an initial segment of $L$ for all $k$, then $N \cdot \omega$ is an initial segment of $L$.
	\end{proposition}
	
Our goal is to prove the following theorem.
	
	\begin{theorem}\label{thm:main-inline}
		Given a satisfiable $\E_{\alpha}$ sentence $T$ of linear orders, there is a linear order $\mc{B} \models \varphi$ with a $\Pi_{\alpha + 3}$ Scott sentence and hence Scott rank at most $\alpha + 2$.
	\end{theorem}
	
	Note that this theorem immediately implies our initially reported Theorem \ref{thm:SmallGaps}.
	In particular, every $\Pi_\alpha$ sentence is also an $\E_{\alpha+1}$ sentence.
	The above theorem then gives a model with a $\Pi_{\alpha + 4}$ Scott sentence, or what is the same a model of Scott rank $\alpha+3$, as reported in Theorem \ref{thm:SmallGaps}.
	
	\subsection{Definitions of Key Properties}
	
	Let $\mathbb{A}$ be a countable existential fragment of $\E_\alpha$ formulas containing $T$ and satisfying the conclusion of Lemma \ref{lem:CountableSet}, together with a class $\mathbb{C}$ of countable structures. Recall that $\mathbb{A}^*$ is the set of satisfiable sentences in $\mathbb{A}$ and that each of them is satisfied by a structure from $\mathbb{C}$.

Given a linear order $\mc{L}$, we say that $\mc{L}$ is \textit{generic} if
\begin{enumerate}
	\item for all $-\infty \leq a < b \leq \infty$ there is a formula $\varphi \in \mathbb{A}^*$ that either forces unity or splitting and such that $\mc{L} \models \varphi \Res_{(a,b)}$, and
	\item for all $-\infty \leq a < b \leq \infty$ and sentences $\psi \in \mathbb{A}^*$, either $\mc{L} \models \psi \Res_{(a,b)}$ or there is $\varphi \in \mathbb{A}^*$ such that $\mc{L} \models \varphi \Res_{(a,b)}$ and $\psi$ and $\varphi$ are incompatible ($\varphi \wedge \psi$ is not satisfiable).
\end{enumerate}
We also isolate the following property of a linear ordering $\mc{L}$, which we call Property $\mathbf{(*)}$.

\begin{description}
	\item[($*$)] $\forall ~ a,b,c,d\in\mc{L}\cup\{-\infty,\infty\}$ with $-\infty \leq a < b \leq c < d \leq \infty$, either:
\begin{description}
	\item[(a)] there are formulas $\psi_1,\psi_2 \in \mathbb{A}^*$ such that $\mc{L} \models \psi_1 \Res_{(a,b)}$ and $\mc{L} \models \psi_2 \Res_{(c,d)}$ but $\psi_1(x,y) \wedge \psi_2(x,y)$ is not satisfiable.
	\item[(b)] there is a formula $\varphi \in \mathbb{A}^*$ that forces unity such that $\mc{L} \models \varphi \Res_{(a,b)} \wedge \varphi \Res_{(c,d)}$.
\end{description} 
\end{description} 
In these properties above, we write, e.g., $-\infty \leq a  < b \leq \infty$ to denote that we include the unbounded intervals $(-\infty,b)$ and $(a,\infty)$. In general, we apply the conventions established previously to always include such intervals, even if not explicitly mentioned.

	\subsection{Verification Stage}

We prove the theorem in two steps.
First, we show that any generic linear order $\mc{L}$ with Property $\mathbf{(*)}$ has Scott rank at most $\alpha+2$.
To show this, we demonstrate that for all $\bar{a},\bar{b}\in\mc{L}$ if $\bar{a} \equiv_{\alpha+1} \bar{b}$, then $\bar{a}$ and $\bar{b}$ are in the same automorphism orbit.
It is known, see for example \cite{MBook} Section II.9, that this gives that $\SR(\mc{L})\leq\alpha+2$ where $\SR$ denotes Montalb\'an's parameterless Scott rank.
In particular, any $\mc{L}$ that is generic and with Property $\mathbf{(*)}$ has a $\Pi_{\alpha+3}$ Scott sentence as desired.

\begin{lemma}\label{lem:main1}
	If $\mc{A}$ is a generic linear ordering with Property $\mathbf{(*)}$, then given any two tuples $\bar{a}$ and $\bar{b}$, if $\bar{a} \equiv_{\alpha+1} \bar{b}$, then $\bar{a}$ and $\bar{b}$ are in the same automorphism orbit. In particular, $\mc{A}$ has a $\Pi_{\alpha + 3}$ Scott sentence.
\end{lemma}

\begin{proof}
Let $\bar{a} = (a_1,\ldots,a_n)$ and $\bar{b} = (b_1,\ldots,b_n)$.
Without loss of generality, we may assume that $a_1 < a_2 < \cdots < a_n$ and $b_1 < b_2 < \cdots < b_n$.
As $\bar{a} \equiv_{\alpha+1} \bar{b}$, we have $(a_i,a_{i+1}) \equiv_{\alpha+1} (b_i,b_{i+1})$ for $i = 0,\ldots,n$.
To show that $\bar{a}$ and $\bar{b}$ are in the same automorphism orbit, it suffices to show that $(a_i,a_{i+1}) \cong (b_i,b_{i+1})$ for each $i$.
Thus it suffices to show that given two intervals $(a,b)$ and $(c,d)$ in $\mc{A}$ with $(a,b) \equiv_{\alpha+1} (c,d)$, $(a,b) \cong (c,d)$.
Without loss of generality, we will assume throughout that $a\leq c$. 

We will make use of several lemmas that describe what happens in an interval depending on whether any subinterval satisfies a formula forcing unity or whether there is a subinterval satisfying a formula forcing splitting. In all of the lemmas to follow, we include intervals with limits $-\infty$ or $\infty$.

\begin{sublemma}\label{IntervalTransfer}
	Let $\varphi \in \mathbb{A}^*$ be a sentence that forces unity. Let $\mc{A}$ and $\mc{B}$ be generic linear orders. Let $(a,b)$ be an interval in $\mc{A}$ and $(c,d)$ an interval in $\mc{B}$ such that $\mc{A} \models \varphi \Res_{(a,b)}$ and $\mc{B} \models \varphi \Res_{(c,d)}$. Suppose that for every subinterval $(a',b')$ of $(a,b)$, $(a',b')$ satisfies a sentence that forces unity. Then, every subinterval $(c',d')$ of $(c,d)$ satisfies a sentence that forces unity.
\end{sublemma}
\begin{proof}
	For the sake of contradiction, say that there is a subinterval $(c',d')$ of $(c,d)$ that does not force unity.
	By genericity, $(c',d')$ must satisfy a sentence $\psi$ that forces splitting.
	This means the sentence $\exists x < y ~ \psi\Res_{(x,y)}$ is consistent with $\varphi$.
	In fact, because $\varphi$ forces unity, $\exists x < y ~ \psi\Res_{(x,y)}$ is also consistent with any $\chi\in\mathbb{A}^*$ that extends $\varphi$.
	By genericity, this means that $(a,b)$ must satisfy $\exists x < y ~ \psi\Res_{(x,y)}$.
	Let $a',b'$ witness the above existential quantifier.
	Note that $(a',b')$ is a subinterval of $(a,b)$, so it satisfies some $\varphi'$ that forces unity.
	However, this means that it is consistent to satisfy both $\psi$ and $\varphi'$ and, therefore to force both unity and splitting simultaneously, a contradiction.
\end{proof}

The above lemma together with Lemma \ref{ScottImpliesIso} to follow justifies the importance of the following definition. 

\begin{definition}
	We say that a sentence $\varphi$ is \emph{Scott-like} if $\varphi$ forces unity and for all $\mc{L} \models \varphi$, for every subinterval $(a,b)$ of $\mc{L}$, $(a,b)$ satisfies a sentence that forces unity (or, equivalently by the previous lemma, there is $\mc{L} \models \varphi$ such that for every subinterval $(a,b)$ of $\mc{L}$, $(a,b)$ satisfies a sentence that forces unity). We say that a linear order $\mc{L}$ is \emph{Scott} if there is a Scott-like sentence $\varphi$ such that $\mc{L} \models \varphi$ (i.e., $\mc{L}$ has a Scott-like sentence).
\end{definition}

We use the terminology Scott-like because, among generic linear orderings, Scott-like sentences act like Scott sentences.
In addition, Scott linear orderings act like linear orderings with a small Scott rank.

\begin{sublemma}\label{ScottImpliesIso}
	Let $\varphi \in \mathbb{A}^*$ be Scott-like. Let $\mc{A}$ and $\mc{B}$ be generic linear orders. Then for all intervals $(a,b)$ in $\mc{A}$ and $(c,d)$ in $\mc{B}$, if $(a,b) \models \varphi$ and $(c,d) \models \varphi$ then $(a,b) \cong (c,d)$. 
\end{sublemma}
\begin{proof} 
	We will build an isomorphism $(a,b) \cong (c,d)$ using a back-and-forth argument. Recall that a back-and-forth family is a set $\mc{F}$ of tuples $(\bar{u},\bar{v})$ such that if $(\bar{u},\bar{v}) \in \mc{F}$ then for any $u'$ there is a $v'$ (and for any $v'$ there is a $u'$) such that $(\bar{u}u',\bar{v}v') \in \mc{F}$. We claim that the following set $\mc{F}$ is a back-and-forth family between $(a,b)$ and $(c,d)$ and thus $(a,b) \cong (c,d)$. Let $\mc{F}$ consist of tuples $(\bar{u},\bar{v})$ which, rearranged in increasing order $u_1 < \cdots < u_n$ in $(a,b)$ and $v_1<\ldots<v_n \in (c,d)$, satisfy that for each $i$ the intervals $(u_i,u_{i+1})$ and $(v_i,v_{i+1})$ both satisfy a common sentence $\varphi_i \in \mathbb{A}^*$ that forces unity. (As usual, we allow $i=0$ and $i = n$, giving the intervals $(a,u_1)$ and $(c,v_1)$, and $(u_n,b)$ and $(v_n,d)$.) Given $u'$, say $u_i < u' < u_{i+1}$, let $\varphi_i \in \mathbb{A}^*$ be the sentence above. By hypothesis, there are sentences $\psi,\theta \in \mathbb{A}^*$ which force unity such that $(u_i,u') \models \psi$ and $(u',u_{i+1}) \models \theta$. Then as witnessed by $u'$, we have $(u_i,u_{i+1}) \models \exists x \; \psi \Res_{<x} \wedge \theta \Res_{>x}$. Since this is a formula of $\mathbb{A}^*$ consistent with $\varphi_i$, by genericity we get that $(v_i,v_{i+1}) \models \exists x \; \psi \Res_{<x} \wedge \theta \Res_{>x}$. Let $v'$ be a witness; then $(v_i,v') \models \psi$ and $(v',v_{i+1}) \models \theta$. Thus $(\bar{u}u',\bar{v}v') \in \mc{F}$.
\end{proof}

\begin{sublemma}\label{finsum}
	Let $(a,b)$ and $(c,d)$ be Scott. Then:
	\begin{enumerate}
		\item The sum $(a,b) + 1 + (c,d)$ is Scott, so in particular any finite sum $(a,b)+1+(a,b)+1+\cdots+1+(a,b)$ is Scott.
		\item Any subinterval of $(a,b)$ is Scott.
	\end{enumerate}
\end{sublemma}
\begin{proof}
	We begin by proving item (1).
	Any subinterval of $(a,b) + 1 + (c,d)$ that is not simply a subinterval of $(a,b)$ or $(c,d)$ is of the form $(p,b)+1+(c,q)$ for some $p\in(a,b)$ and $q\in(c,d)$.
	Let $\psi_1$ be the sentence that forces unity for $(p,b)$ and let $\psi_2$ be the sentence that forces unity for $(c,q)$.
	Consider the following sentence $\theta\in\mathbb{A}^*$
	\[\theta : = ~ \exists x ~ \psi_1\Res_{< x}\land  \psi_2\Res_{> x}.\]
	We claim that $\theta$ forces unity for $(p,b)+1+(c,q)$.
	If it does not, there are inconsistent extensions $\chi_1$ and $\chi_2$ of $\theta$ in $\mathbb{A}^*$.
	Pick a model $\mc{M}_1 + 1 + \mc{N}_1\in\mathbb{C}$ of $\theta \wedge \chi_1$ with $\mc{M}_1 \models \psi_1$ and $\mc{N}_1 \models \psi_2$, and a model $\mc{M}_2 + 1 + \mc{N}_2\in \mathbb{C}$ of $\theta \wedge \chi_2$ with $\mc{M}_2 \models \psi_1$ and $\mc{N}_2 \models \psi_2$. By Lemma \ref{lem:intervalFormulas} we can find $\rho_{k,j}$ in $\mathbb{A}^*$ with $1\leq k\leq 2$ and $1\leq j\leq 2$ such that:
	\begin{align*}
		\mc{M}_1 \models \rho_{1,1} && \mc{N}_1 \models \rho_{1,2} \\
		\mc{M}_2 \models \rho_{2,1} && \mc{N}_2 \models \rho_{2,2}
	\end{align*}
	and such that if
	\[ \mc{L} \models \exists x \; \big(\rho_{k,1} \Res_{< x}\wedge \rho_{k,2}\Res_{> x}\big) \]
	then $\mc{L} \models \chi_k$.
	In other words,
	\[ \models \exists x \big(\rho_{k,1} \Res_{< x}\wedge \rho_{k,2}\Res_{> x}\big) \longrightarrow \chi_k. \]
	Note that if for each $j$ if the sentence $\rho_{1,j}$ is consistent with $\rho_{2,j}$, say with model $\mc{L}_j$, then we would have $\mc{L}_1 + 1 + \mc{L}_2 \models \chi_1 \wedge \chi_2$.
	This means that for some $j$, $\rho_{1,j}$ is inconsistent with $\rho_{2,j}$, and both are consistent with $\psi_j$.
	This is a contradiction to the fact that $\psi_1$ and $\psi_2$ force unity. Thus $(a,b)+1+(c,d)$ is Scott. By straightforward induction we can prove the second statement of the lemma starting ``in particular...''.
	
	We now prove item (2).
	Being Scott is hereditarily defined based on the properties of all of the subintervals of $(a,b)$.
	As $(c,d)$ is a subinterval of $(a,b)$ all subintervals of $(c,d)$ are also subintervals of $(a,b)$.
	This means that if $(a,b)$ is Scott then $(c,d)$ must be Scott.
\end{proof}

We will now explain what happens for intervals that are not Scott in a generic linear ordering. This is where we use ($*$).

\begin{sublemma}\label{splittingFree}
	Given a generic linear order satisfying ($*$), suppose that $a <b$ and $c < d$ and that $(a,b)$ is not Scott as witnessed by some subinterval $(a',b')$ that is disjoint from $(c,d)$. Then $(a,b) \nequiv_{\alpha+1} (c,d)$.
\end{sublemma}
\begin{proof}
	Suppose towards a contradiction that $(a,b) \equiv_{\alpha+1} (c,d)$. Let $\varphi$ be the sentence true of $(a',b')$ which forces splitting. By ($*$), for every subinterval $(c',d')$ of $(c,d)$, with $c \leq c' < d' \leq d$ and with $\mc{L} \models \varphi \Res_{(c',d)}$, there is an $\E_\alpha$ sentence $\psi_{c',d'} \leq \varphi$ such that $(c',d') \models \psi_{c',d'}$ and such that $(a',b') \nmodels \psi_{c',d'}$. Then
	\[ \mc{L} \models \varphi \Res_{(a',b')} \wedge \bigdoublewedge_{\substack{ < c' < d' < d \\ \mc{L} \models \varphi \Res_{(c',d')}}} \neg \psi_{c',d'} \Res_{(a',b')}.\]
	In particular,
	\[ (a,b) \models \exists x < y \;\;\; \left[\varphi \Res_{(x,y)} \wedge \bigdoublewedge_{\substack{c < c' < d' < d \\ \mc{L} \models \varphi \Res_{(c',d')}}} \neg \psi_{c',d'} \Res_{(x,y)} 
	\right] \]
	but
	\[ (c,d) \models \forall x < y \;\;\; \left[ \varphi \Res_{(x,y)} \;\longrightarrow\; \bigdoublevee_{\substack{c < c' < d' < d \\ \mc{L} \models \varphi \Res_{(c',d')}}} \psi_{c',d'} \Res_{(x,y)}\right].\]
	This sentence on which $(a,b)$ and $(c,d)$ disagree is a $\E_{\alpha+1}$/$\A_{\alpha+1}$ formula. Thus, using Corollary \ref{cor:karp}, $(a,b) \nequiv_{\alpha+1} (c,d)$.
\end{proof}

\medskip

We now use the lemmas to finish the argument. Recall that our goal is to show that given two intervals $(a,b)$ and $(c,d)$ in $\mc{A}$ with $(a,b) \equiv_{\alpha+1} (c,d)$, $(a,b) \cong (c,d)$. We may assume, without loss of generality, that $a \leq c$. There will be a few different cases, depending on how $b$ and $d$ compare to each other and to $c$ and whether or not certain intervals are Scott.

In the case that $a < b \leq c < d$, we may directly apply Lemmas \ref{ScottImpliesIso} and \ref{splittingFree}. If both of the intervals $(a,b)$ and $(c,d)$ are Scott, then as $(a,b) \equiv_{\alpha+1} (c,d)$ they have the same Scott-like sentence and hence are isomorphic. Otherwise, if they are not Scott, Lemma \ref{splittingFree} would imply that $(a,b) \nequiv_{\alpha+1} (c,d)$, contradicting our original assumption.

Otherwise, $a\leq c<b$ so that the two intervals $(a,b)$ and $(c,d)$ overlap. We first consider the subcase that the interval $(a,c)$ is not Scott as witnessed by $(a',b')$. Note that $(a',b')$ is disjoint from $(c,d)$ as it is contained in $(a,c)$. Lemma \ref{splittingFree} gives that $(a,b) \nequiv_{\alpha+1} (c,d)$, a contradiction to our assumption. Note that a symmetrical argument handles the case that either $(b,d)$ if $b < d$, or $(d,b)$, if $d < b$, is not Scott. Thus, we may assume that the intervals $(a,c)$ and either $(b,d)$ or $(d,b)$ are Scott. We split into the two cases $a\leq c<b\leq d$ and $a\leq c< d \leq b$.
	
First, consider the case where $a\leq c<b\leq d$ and the intervals $(a,c)$ and $(b,d)$ are Scott. Let $N$ be the order type of $(a,c)$ and let $\psi$ be the Scott-like sentence for $(a,c)$. If $(a,b)$ is Scott, then so is $(c,d)$, and $(a,b) \cong (c,d)$ by Lemma \ref{ScottImpliesIso} (and vice versa). So, we may assume that $(a,b)$ and $(c,d)$ are not Scott. 

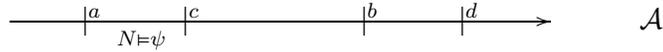
\[
\xymatrix@R=0.1pc{
\ar[rrrrrr] & |^a & |^c & & |^{b} & |^d &   & \mc{A} \\
& \ar@{}^{N\models \psi}[r]  & & & & & &
} \]

We will argue that $(N + 1) \cdot \omega$ is an initial interval of $(c,b)$. By Lemma \ref{Tarski}, to show that $(N + 1) \cdot \omega$ is an initial interval of $(c,b)$, it suffices to show that $(N + 1) \cdot k$ is an initial segment of $(c,b)$ for every $k$. Suppose that, starting with $k = 0$, $(N + 1) \cdot k$ is an initial segment of $(c,b)$; we will argue that $(N + 1) \cdot (k+1)$ is an initial segment of $(c,b)$. Let $c = u_0 < u_1 < u_2 < \cdots < u_k$ be such that $(u_{i-1},u_{i}) \cong N$ for $i = 1,\ldots,k$. Note that also $(a,c) \cong N$. Then, as witnessed by $x_i = u_{i-1}$,
\[ (a,b) \models \exists x_1 < \cdots < x_{k+1} \;\; \bigwedge_{i=1}^n (x_{i-1},x_i) \models \psi \]
where, by convention, $x_0$ is $a$/$-\infty$. This is $\E_\alpha$, and so
\[ (c,d) \models \exists x_1 < \cdots < x_{k+1} \;\; \bigwedge_{i=1}^n (x_{i-1},x_i) \models \psi.\]
Let $v_0 = c < v_1 < \cdots < v_{k+1}$ be the witnesses; since each interval satisfies $\psi$, we have $(v_{i-1},v_i) \cong N$ for each $i = 1,\ldots,k+1$ by Lemma \ref{ScottImpliesIso}. Thus $(N+1) \cdot (k+1)$ is an initial segment of $(c,d)$. It must be that each $v_i \in (c,b)$, and $(N+1) \cdot (k+1)$ is an initial interval of $(c,b)$ as desired; otherwise, $(c,b)$ is an initial interval of $(N+1) \cdot (k+1)$, and so would be Scott by Lemma \ref{finsum}. Since we assumed this was not the case, we are done and have shown that $(N+1) \cdot \omega$ is an initial interval of $(c,b)$. 

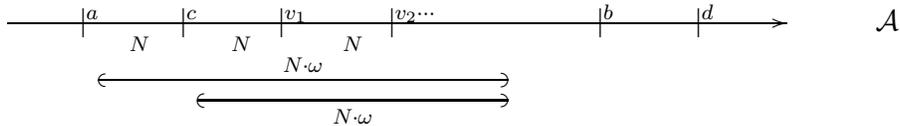
\[
\xymatrix@R=0.1pc{
\ar[rrrrrrrr] & |^a & |^c & |^{v_1} & |^{v_2\cdots} && |^{b} & |^d &   & \mc{A} \\
& \ar@{}^{N}[r] & \ar@{}^{N}[r]  & \ar@{}^{N}[r]  & & \\
& \ar@{(-)}^{N\cdot\omega}[rrrr] &&&& \\
& & \ar@{(-)}_{N\cdot\omega}[rrr] &&&&
} \]

Thus $(a,b) \cong N + 1 + (c,d) \cong (c,b)$. A similar argument on the other side, using the fact that $(b,d)$ is Scott, shows that $(c,d) \cong (c,b)$. Thus $(a,b) \cong (c,d)$.
To be explicit, note that nothing changes in this argument if we let some of the end points be $\infty$ or $-\infty$.

Finally, we consider the case where $a\leq c<d<b$ and the intervals $(a,c)$ and $(d,b)$ are Scott. We use a variation on the above argument. Let $N$ be the isomorphism type of $(a,c)$ and $\psi$ the Scott-like sentence. We argue that $(N+1) \cdot \omega$ is an initial interval of $(c,d)$. Again, given that $(N+1) \cdot k$ is an initial interval of $(c,d)$, we show that $(N+1) \cdot (k+1)$ is. We have
\[ (a,b) \models \exists x_1 < \cdots < x_{k+1} \;\; \bigwedge_{i=1}^n  \psi \Res_{(x_{i-1},x_i)} \]
where, by convention, $x_0$ is $a$/$-\infty$. This is $\E_\alpha$, and so
\[ (c,d) \models \exists x_1 < \cdots < x_{k+1} \;\; \bigwedge_{i=1}^n \psi \Res_{(x_{i-1},x_i)}.\]
Thus $(N+1) \cdot (k+1)$ is an initial interval of $(c,d)$ by Lemma \ref{ScottImpliesIso}. We conclude that $(N+1) \cdot \omega$ is an initial interval of $(c,d)$, and so $(a,d) \cong (c,d)$. Similarly, on the other side, we argue that $(c,b) \cong (c,d)$. Then $(a,b) \cong (a,d) + 1 + (d,b) \cong (c,d) + 1 + ( d,b) \cong (c,b) \cong (c,d)$.

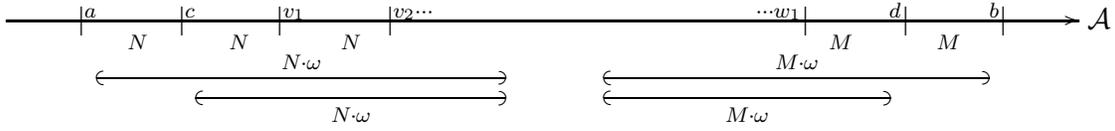
\[
\xymatrix@R=0.1pc{
\ar[rrrrrrrrrrr] & |^a & |^c & |^{v_1} & |^{v_2\cdots} &&&& {}^{\cdots w_1}| & {}^{d}| &  {}^{b}| &  \mc{A} \\
& \ar@{}^{N}[r] & \ar@{}^{N}[r]  & \ar@{}^{N}[r]  &&&&& \ar@{}^{M}[r] & \ar@{}^{M}[r] &  \\
& \ar@{(-)}^{N\cdot\omega}[rrrr] &&&&& \ar@{(-)}^{M\cdot\omega}[rrrr] &&&&& \\
& & \ar@{(-)}_{N\cdot\omega}[rrr] &&&& \ar@{(-)}_{M\cdot\omega}[rrr] &&&&&
} \]

\end{proof}

\subsection{Construction Stage}

We now demonstrate that it is always possible to find a generic linear order with Property $\mathbf(*)$ satisfying the $\E_\alpha$ sentence $T$.

\begin{lemma}\label{lem:main2}
Given the satisfiable $\E_\alpha$ sentence $T$ of linear orders, there is a generic linear ordering $\mc{A}$ with Property $\mathbf(*)$ such that $\mc{A}\models \varphi$.
\end{lemma}

\begin{proof}
	We use the same pair $\mathbb{A}$, $\mathbb{C}$ that was fixed above. We note once more than when considering intervals involving constants, we allow $\infty$ and $-\infty$. Let $C = \{c_0,c_1,\ldots\}$ be a new set of constants. Using a Henkin-style construction we build a set $S$ of $\E_{\alpha}$ sentences over the language $\{ < \} \cup C$, with each sentence being a substitution of the constants $C$ into a formula of $\mathbb{A}$, such that:
	\begin{enumerate}
		\item If $\bigdoublevee \psi_i \in S$, then for some $i$, $\psi_i \in S$.
		\item If $\exists \bar y \psi(\bar y) \in S$, then $\psi(\bar c) \in S$ for some constants $\bar c \in C$.
		\item If $\bigdoublewedge \psi_i \in S$, then for all $i$, $\psi_i \in S$.
		\item If $\forall \bar y \psi(\bar y) \in S$, then $\psi(\bar c) \in S$ for all $\bar c \in C$.
		\item For every atomic sentence $\psi$ over $\tau \cup C$, either $\psi \in S$ or $\neg \psi \in S$.
		\item The model is generic:
		\begin{enumerate}
			\item for any set of constants $a,b$ with $a < b$ in $S$ there is a a sentence $\varphi \in \mathbb{A}^*$ that either forces unity or splitting and such that $\varphi \Res_{(a,b)} \in S$, and
			\item for all constants $a,b$ with $a < b$ in $S$ and sentences $\psi \in \mathbb{A}^*$, either $\psi \Res_{(a,b)} \in S$ or there is $\varphi \in \mathbb{A}^*$ such that $\varphi \Res_{(a,b)} \in S$ and $\psi$ and $\varphi$ are incompatible.
		\end{enumerate}
		\item Property ($*$): For any set of constants $a,b,c,d$ with $a < b \leq c < d \in S$, either (a)  there are formulas $\psi_1,\psi_2 \in \mathbb{A}^*$ such that $\psi_1 \Res_{(a,b)} \in S$ and $\psi_2 \Res_{(c,d)} \in S$ but $\psi_1\Res_{(x,y)} \wedge \psi_2 \Res_{(x,y)}$ is not satisfiable, or (b) there is a formula $\varphi \in \mathbb{A}^*$ that forces unity such that $\varphi \Res_{(a,b)} \wedge \varphi \Res_{(c,d)} \in S$.
	\end{enumerate}
	Once we have such a set $S$, we can build a structure $\mc{A}$ as usual. As indicated above, Property (6) will guarantee that $\mc{A}$ is generic, and Property (7) will guarantee that $\mc{A}$ has Property ($*$). As usual, the sentences in $S$ will hold of the constants mentioned in them. In particular, if we guarantee that $T\in S$ (as we will), $\mc{A}\models T$ as desired. Therefore, we only need to construct such a set $S$ to complete the proof.

	We build $S$ stage-by-stage, with at each stage $S_n$ consisting of finitely many $\E_{\alpha}$ sentences in the language $\{<\} \cup C$, with each sentence mentioning at most finitely many of the constants from $\bar c$. Then, in the end, we let $S = \bigcup S_n$. We will make sure that at each stage $n$ there is a linear order $\mc{B}\in\mathbb{C}$ and an assignment $v_n$ of values from $\mc{B}$ to the constants that appear in $S_n$, such that $S_n$ holds in $\mc{B}$. Moreover, for any two constants $c$ and $d$ which appear in $S_n$, either $c=d$ or $c < d$ or $d < c$ will be in $S_n$. We begin with $S_0 = \{ T\}$.

	At each stage, we take care of a new instance of one of the requirements. For dealing with (1)-(5), we do not have to change the structure $\mc{B}$ which witnessed the satisfiability at the previous stage. In other words, these items are dealt with in a standard way (see, e.g., \cite{MonSR}). In any case, we will describe how to meet these requirements.  
	
	Suppose that at stage $n+1$ we must deal with an instance of one of these requirements. We assign to each instance infinitely many stages; if, e.g., $\bigdoublevee \psi_i \in S_n$ and $n+1$ is a stage at which we are to deal with the corresponding instance of (1), then we must ensure that for some $i$, $\psi_i \in S_{n+1}$. For (3) and (4), the instances are not just for the particular formula $\bigdoublewedge \psi_i$ or $\forall \bar{y} \psi(\bar{y})$, but also for the $k$ for which we must put $\psi_k \in S$, or the constant $\bar{c}$ for which we must put $\psi(\bar{c}) \in S$. (That is, we will not put all $\psi_k$ or $\psi(\bar{c})$ in $S$ at one time, but one by one over the whole construction.) 
	
	\begin{enumerate}
		\item Suppose that we deal with an instance of (1) corresponding to $\bigdoublevee \psi_i \in S_n$ at stage $n+1$. There is a linear order $\mc{B}\in\mathbb{C}$ and an assignment $v_n$ of values from $\mc{B}$ to the constants that appear in $S_n$ such that $S_n$ holds in $\mc{B}$. In particular, some $\psi_i$ must hold in $\mc{B}$. Let $S_{n+1} = S_n \cup \{\psi_i\}$ and keep $v_n$ and $\mc{B}$ unchanged.
	
		\item For an instance of (2) corresponding to $\exists \bar y \psi(\bar y) \in S$. There is a linear order $\mc{B}\in\mathbb{C}$ and an assignment $v_n$ of values from $\mc{B}$ to the constants that appear in $S_n$ such that $S_n$ holds in $\mc{B}$, and in particular, $\exists \bar y \psi(\bar y)$ holds in $\mc{B}$. Thus, there must be witnessing elements $\bar{b}$. Choosing a new tuple of constants $\bar{d}$ not appearing in $S_n$, let $S_{n+1}$ be $S_n$ together with $\psi(\bar{d})$ as well as the sentences describing the ordering of $\bar{d}$ compared to the constants already appearing in $S_n$ (as obtained from $\bar{b}$). This is satisfied in $\mc{B}$ under the assignment $v_{n+1}$ which modifies $v_n$ by assigning the tuple $\bar{b}$ to the constants $\bar{d}$.
	
		\item For the $k$th instance of (3) corresponding to $\bigdoublewedge \psi_i \in S_n$ (for which we must put $\psi_k \in S$), there is a linear order $\mc{B}\in\mathbb{C}$ and an assignment $v_n$ of values from $\mc{B}$ to the constants that appear in $S_n$ such that $S_n$ holds in $\mc{B}$, and particular $\psi_k$ holds in $\mc{B}$. Put $S_{n+1} = S_n \cup \{\psi_k\}$ and keep $v_n$ and $\mc{B}$ unchanged.
	
		\item For the instance of (4) corresponding to $\forall \bar{y} \psi(\bar{y})$ and $\bar{c}$, there is a linear order $\mc{B}\in\mathbb{C}$ and an assignment $v_n$ of values from $\mc{B}$ to the constants that appear in $S_n$ such that $S_n$ holds in $\mc{B}$, and in particular, extending this assignment to the new constants in $\bar{c}$, $\psi(\bar{c})$ holds in $\mc{B}$. Let $S_{n+1}$ be $S_n$ together with $\{\psi(\bar{c})\}$ and the sentences describing the ordering of the constants (as obtained from $\mc{B}$).
		
		\item For the instance of (5) corresponding to the atomic sentence $\psi$, there is a linear order $\mc{B}\in\mathbb{C}$ and an assignment $v_n$ of values from $\mc{B}$ to the constants that appear in $S_n$ such that $S_n$ holds in $\mc{B}$. If $\mc{B} \models_{v_n} \psi$ put $S_{n+1} = S_n \cup \{\psi\}$, and otherwise if $\mc{B} \models_{v_n} \neg \psi$ put $S_{n+1} = S_n \cup \{\neg \psi\}$ and keep $v_n$ and $\mc{B}$ unchanged.
	\end{enumerate}
	
	It is the instances of (6), both (a) and (b), and (7) where we must do most of the interesting work.
	
	\begin{enumerate}[resume*]
		\item 
		
		\begin{enumerate}
			\item Let $a< b$ be the constants corresponding to this instance. Let $\theta(a,b,\bar d) = \bigwedge S_n$ where $\bar d$ are the additional constants that appear in $S_n$ other than $a,b$. Then $\exists \bar z \theta(x,y, \bar z)$ is a $\E_{ \alpha}$ formula realized in $\mc{B}\in\mathbb{C}$ by $x=a, y=b$. By Lemma \ref{lem:intervalFormulas} there are $\E_{\alpha}$ sentences $\chi_1,\chi_2,\chi_3$ in $\mathbb{A}$ such that
			\[
				(-\infty,a) \models \chi_1, \quad (a,b) \models \chi_2, \quad (b,\infty) \models \chi_3\]
			and such that if $\mc{L}$ is any linear order with elements $c < d$ and
			\[ \mc{L} \models \chi_1 \Res_{< c} \wedge \chi_2 \Res_{(c,d)} \wedge \chi_3 \Res_{>d} \]
			then $\mc{L} \models \exists \bar z \theta(c,d, \bar z)$.
			There is a formula $\psi \in \mathbb{A}^*$, $\psi \leq \chi_2$, such that $\psi$ either forces unity or forces splitting. 
			Let $\mc{P}\models \psi$.
			Note that $\mc{B}_{\leq a}+\mc{P}+\mc{B}_{\geq b}$ satisfies the formula $\exists x < y ~  \left(\chi_1 \Res_{<x} \wedge \psi \Res_{(x,y)} \wedge \chi_3 \Res_{>y}\right)$.
			As $\psi \leq \chi_2$, this linear ordering also satisfies $\exists \bar z \theta(x,y, \bar z)$.
			By Property (2) of Lemma \ref{lem:CountableSet}, there is a $\mc{B}^*\in\mathbb{C}$ that also satisfies
			\[ \mc{B}^*\models\exists x < y ~  \left(\chi_1 \Res_{<x} \wedge \psi \Res_{(x,y)} \wedge \chi_3 \Res_{>y} \land \exists \bar z \theta(x,y, \bar z)
			\right).\]
			We move to interpreting the constants as elements of $\mc{B}^*$.
			The new values of $a,b,\bar{d}$ are given by the witnesses for $x,y,\bar{z}$ respectively.
			With this assignment, it is consistent to put
			\[ S_{n+1} = S_n \cup \{ \chi_1 \Res_{<a} \wedge \chi_2 \Res_{(a,b)} \wedge \chi_3 \Res_{>b}\} \cup \{\psi \Res_{(a,b)}\},\]
			which has the desired property that $(a,b)$ is assigned a formula that either forces splitting or unity. Also if $a$ and $b$ are new constants, add to $S_{n+1}$ their ordering relative to the other constants.
			Note that this argument only changes superficially if either $a=-\infty$ or $b=\infty$.
			
			\item Let $a< b$ be the constants corresponding to this instance and let $\psi$ be the sentence. Let $\theta(a,b,\bar d) = \bigwedge S_n$ where $\bar d$ are the additional constants that appear in $S_n$ other than $a,b$. Then $\exists \bar z \theta(x,y, \bar z)$ is a $\E_{ \alpha}$ formula realized in $\mc{B}$ by $x=a, y=b$. By Lemma \ref{lem:intervalFormulas} there are $\E_{\alpha}$ sentences $\chi_1,\chi_2,\chi_3$ in $\mathbb{A}$ such that
			\[
			(-\infty,a) \models \chi_1, \quad (a,b) \models \chi_2, \quad (b,\infty) \models \chi_3\]
			and such that if $\mc{L}$ is any linear order with elements $c < d$ and
			\[ \mc{L} \models \chi_1 \Res_{< c} \wedge \chi_2 \Res_{(c,d)} \wedge \chi_3 \Res_{>d} \]
			then $\mc{L} \models \exists \bar z \theta(c,d, \bar z)$.
			If $\psi$ is consistent with $\chi_2$, then there is $\mc{P}\models \psi\land\chi_2$.
			Note that 
			\[\mc{B}_{\leq a}+\mc{P}+\mc{B}_{\geq b}\models\exists x <y \left( \chi_1 \Res_{<x} \wedge \big(\psi\land\chi_2\big) \Res_{(x,y)} \wedge \chi_3 \Res_{>y}\right).\]
			By Property (2) of Lemma \ref{lem:CountableSet}, there is a $\mc{B}^*\in\mathbb{C}$ that also satisfies
			\[ \mc{B}^*\models\exists x < y \left(  \chi_1 \Res_{<x} \wedge \psi\land\chi_2 \Res_{(x,y)} \wedge \chi_3 \Res_{>y} \land \exists \bar z \theta(x,y, \bar z)\right).\]
			 We move to interpreting the constants as elements of $\mc{B}^*$.
			The new values of $a,b,\bar{d}$ are given by the witnesses for $x,y,\bar{z}$ respectively.
			With this assignment, it is consistent to put
			\[ S_{n+1} = S_n \cup \{ \chi_1 \Res_{<a} \wedge \chi_2 \Res_{(a,b)} \wedge \chi_3 \Res_{>b}\} \cup \{\psi \Res_{(a,b)}\}.\]

			Otherwise, $\psi$ is not consistent with $\chi_2$.
			In this case, we keep interpreting our constants in $\mc{B}$ and let 
			\[ S_{n+1} = S_n \cup \{ \chi_1 \Res_{<a} \wedge \chi_2 \Res_{(a,b)} \wedge \chi_3 \Res_{>b}\}.\]
			$\chi_2$ is now the witness to $\psi$ not being satisfied by the interval $(a,b)$. Also if $a$ and $b$ are new constants, add to $S_{n+1}$ their ordering relative to the other constants.
			Note that this argument only changes superficially if either $a=-\infty$ or $b=\infty$.
		\end{enumerate}
		
		\item 
		Let $\bar{c} = (c_1,c_2,c_3,c_4)$ be the set of constants for this instance of (7).
		Let $\theta(\bar c,\bar d) = \bigwedge S_n$ where $\bar d$ are the additional constants that appear in $S_n$ but are not already in $\bar c$.
		
		Then $\exists \bar y \theta(\bar x, \bar y)$ is a $\E_{\alpha}$ formula realized in $\mc{B}\in\mathbb{C}$ by $\bar{x} = \bar{c}$. By Lemma \ref{lem:intervalFormulas}, there are $\E_{\alpha}$ formulas $\chi_1,\chi_2,\chi_3,\chi_4,\chi_5$ such that
		\[ \mc{B} \models \chi_1 \Res_{(-\infty,c_1)} \wedge \chi_2 \Res_{(c_1,c_2)} \wedge \cdots \wedge \chi_5 \Res_{(c_4,\infty)}\]
		and such that if $\mc{L}$ is a linear order with elements $d_1 < \cdots < d_4$ and
		\[ \mc{L} \models \chi_1 \Res_{<d_1} \wedge \chi_2 \Res_{(d_1,d_2)} \wedge \cdots \wedge \chi_5 \Res_{> d_4}\]
		then $\mc{L} \models \exists \bar y \theta(\bar x, \bar y)$.
		Consider the case where there are $\E_{\alpha}$ sentences $\psi_1$ and $\psi_2$ such that $\psi_1 \wedge \psi_2$ is inconsistent but $\chi_2 \wedge \psi_1$ and $\chi_4 \wedge \psi_2$ are consistent.
		Then, there is $\mc{P}\models \psi_1\land\chi_2$ and $\mc{Q}\models \psi_2\land\chi_4$.
		Note that,
		\begin{align*}\mc{B}_{\leq c_1}+\mc{P}+[c_2,c_3]+ \mc{Q}+\mc{B}_{\geq b}\models\exists x_1,&x_2,x_3,x_4 \;\;\big( \chi_1 \Res_{<x_1} \wedge \big(\psi_1\land\chi_2\big) \Res_{(x_1,x_2)}  \\& \wedge \chi_3 \Res_{(x_2,x_3)}\wedge \big(\psi_2\land\chi_4\big)\Res_{(x_3,x_4)}\wedge \chi_5\Res_{> x_4}\big).\end{align*}
		By Property (2) of Lemma \ref{lem:CountableSet}, there is a $\mc{B}^*\in\mathbb{C}$ that also satisfies this formula.
		We move to interpreting the constants as elements of $\mc{B}^*$.
		We let $\bar{c}$ be the witnesses to $x_1,x_2,x_3,x_4$.
		As the intervals satisfy the $\chi_i$, we know $\mc{B}^*\models \exists \bar{y} \theta(\bar c,\bar y)$.
		Call $\bar{d}$ the witnesses to $\bar{y}$.
		With this assignment, it is consistent to put
		\[ S_{n+1} = S_n \cup \{\chi_1 \Res_{<c_1}, \chi_2 \Res_{(c_1,c_2)}, \chi_3 \Res_{(c_2,c_3)}, \chi_4 \Res_{(c_3,c_4)}, \chi_5 \Res_{> c_4}\} \cup \{\psi_1 \Res_{(c_1,c_2)}, \psi_2 \Res_{(c_3,c_4)}\}.\]
		This forever ensures that the quadruple $\bar{c} = (c_1,c_2,c_3,c_4)$ is in case (a) of (7).
		
		If there are no such $\E_{\alpha}$ sentences $\psi_1$ and $\psi_2$, then first of all $\chi_2 \wedge \chi_4$ must be consistent.
		Moreover, for every $\E_{\alpha}$ sentences $\psi_1$ and $\psi_2$ in $\mathbb{A}^*$ which are consistent with $\chi_2 \wedge \chi_4$, the sentence $\chi_2 \wedge \chi_4 \wedge \psi_1 \wedge \psi_2$ must be consistent.
		This means that $\chi_2 \wedge \chi_4$ forces unity. 
		There is $\mc{P}\models \chi_2\land\chi_4$.
		Note that,
		\begin{align*}\mc{B}_{\leq c_1}+\mc{P}+[c_2,c_3]+ \mc{P}+\mc{B}_{\geq b}\models\exists x_1,&x_2,x_3,x_4 \;\;\big(
		\chi_1 \Res_{<x_1} \wedge \big(\chi_4\land\chi_2\big) \Res_{(x_1,x_2)} \\&\wedge \chi_3 \Res_{(x_2,x_3)}\wedge \big(\chi_2\land\chi_4\big)\Res_{(x_3,x_4)}\wedge \chi_5\Res_{>x_4}\big).\end{align*}
		By Property (2) of Lemma \ref{lem:CountableSet}, there is a $\mc{B}^*\in\mathbb{C}$ that also satisfies this formula.
		We move to interpreting the constants as elements of $\mc{B}^*$.
		We let $\bar{c}$ be the witnesses to $x_1,x_2,x_3,x_4$.
		As the intervals satisfy the $\chi_i$, we know $\mc{B}^*\models \exists \bar{y} \theta(\bar c,\bar y)$.
		Call $\bar{d}$ the witnesses to $\bar{y}$.
		With this assignment, it is consistent to put
		\[ S_{n+1} = S_n \cup \{\chi_1 \Res_{<c_1}, (\chi_2 \wedge \chi_4) \Res_{(c_1,c_2)}, \chi_3 \Res_{(c_2,c_3)}, (\chi_2 \wedge \chi_4) \Res_{(c_3,c_4)}, \chi_5 \Res_{>c}\}.\]
		This ensures that the quadruple $\bar{c} = (c_1,c_2,c_3,c_4)$ is in case (b) of (7).
		
		In both cases, we should also add to $S_{n+1}$ the ordering of any new constants. Note that this argument only changes superficially if $c_1=-\infty$ and/or $c_4=\infty$.
	\end{enumerate}
	This ends the construction, proving the lemma.
\end{proof}

Putting the full construction together, given the satisfiable $\E_\alpha$ sentence $T$ of linear orders, by Lemma \ref{lem:main2} there is a generic linear ordering $\mc{A}$ with Property $\mathbf(*)$ such that $\mc{A}\models \varphi$.  By Lemma \ref{lem:main2} given any two tuples $\bar{a}$ and $\bar{b}$ in $\mc{A}$, if $\bar{a} \equiv_{\alpha+1} \bar{b}$, then $\bar{a}$ and $\bar{b}$ are in the same automorphism orbit. In particular, $\mc{A}$ has a $\Pi_{\alpha + 3}$ Scott sentence.

\section{Examples of Scott Skipping in Linear Orderings}\label{exampleLOs}

Our main result indicates that the complexity of a theory of linear orderings and its simplest model is small and finite.
In this section, we construct some examples of a non-trivial gap between the complexity of a theory of linear orderings and its simplest model.
In particular, we construct examples at limit and non-limit levels of $\Pi_\alpha$ sentences with structures of Scott rank at least $\alpha+1$.
This means that linear orderings do not have the strongest tameness property once conjectured in Question \ref{question:BIRS} (i.e. that all $\Pi_\alpha$ sentences have a model of Scott rank $\alpha$).
We demonstrate examples of this phenomenon at limit ordinals and away from limit ordinals.
In the subsequent section, we will demonstrate that we can take $\alpha$ to be essentially any countable ordinal by iterating our constructions up the hyperarithmetic hierarchy. 
We do not construct an example that is optimal in the sense that it is exactly an $\E_\alpha$ sentence with models of Scott rank at least $\alpha+2$.
This is left as an open question.

\subsection{Scott Skipping at Limit Ordinals}

In \cite{GHTH} the authors together with Ho give constructions for linear orders of Scott complexity $\Sigma_{\lambda + 1}$ where $\lambda$ is a limit ordinal. Moreover, there are structures $\mc{A}$ such that all $\mc{B} \equiv_\lambda \mc{A}$ have Scott complexity $\Sigma_{\lambda + 1}$. By taking an $\omega$-sum of one of these examples, one gets:

\begin{theorem}[Gonzalez, Harrison-Trainor, and Ho, Corollary 6.11  of \cite{GHTH}]
	Let $\lambda$ be a limit ordinal. There is a $\Pi_\lambda$ theory $T$ extending the axioms of linear orders such that no model of $T_\lambda$ has a $\Sigma_{\lambda + 2}$ Scott sentence. In particular, all models of $T_\lambda$ have Scott rank $\geq \lambda + 1$.
\end{theorem}

\noindent Though our interest in this theorem was due to the work in this paper, the proof appears in \cite{GHTH} since it relies in large part on the constructions, definitions, and lemmas there.

\subsection{Almost-bi-interpretations}

As referred to in Proposition \ref{prop:must-use-new2} earlier, and in other constructions below, it will be convenient to consider simpler constructions (e.g., in linear order with unary relations) and then transfer results to more complicated structures. We take a brief interlude to introduce almost-bi-interpretations that formalize these transfer results. We begin with an important definition from \cite{MonRice}.

\begin{definition}
	Given a class of $\tau$-structures $\mathbb{C}$ with only countably many back-and-forth $\alpha$-equivalence classes for tuples, let $(\bar{a}_i)_{i\in\omega}$ be an enumeration of representatives of these classes.
	Define $\tau_{(\alpha)}$ to be the language $\tau$ along with relations $R_i$ of arity $\vert \bar{a}_i \vert$ for each $i$.
	In each structure $\mc{A}\in\mathbb{C}$ let $\mc{A}_{(\alpha)}$ be the $\tau_{(\alpha)}$ structure where $R^i_{\mc{A}}(\bar{b})\iff \bar{a}_i\leq_\alpha \bar{b}$.
	We call $\mc{A}_{(\alpha)}$ the \textit{the canonical $\alpha$-jump of the structure}.
	We will let $\mathbb{C}_{(\alpha)}$ denote the set of the canonical $\alpha$-jumps of the structures in $\mathbb{C}$.
\end{definition}

\cite{MR23} provides tools to demonstrate how this concept interacts with effective bi-interpretation.
We show that the results in Section 2.2 of \cite{MR23} can be generalized to a setting that is more favorable for our uses.

\begin{proposition}\label{prop:DinoAntonio}
	If 
	\begin{itemize}
		\item $\mathbb{D}$ is computably interpretable in $\mathbb{C}_{(\alpha)}$ via $\mc{A}\mapsto(\Phi(\mc{A}))_{(\alpha)}$,
		\item  and $\mathbb{C}_{(\alpha)}$ is computably interpretable in $\mathbb{D}$ via $\mc{B}\mapsto(\Psi(\mc{B}))_{(\alpha)}$,
		\item $\Phi$ and $\Psi$ are inverses of each other up to isomorphism,
		\item the closure, under isomorphism, of the image of $\Phi$ is $\Pi_\gamma$,
	\end{itemize}
	then
	\begin{enumerate}
		\item For every $\Pi_{\alpha+\beta}$ formula $\varphi$ there is a $\Pi_{\beta}$ formula $\varphi^*$ such that $\mc{M}\models \varphi^* \iff \Phi(\mc{M})\models \varphi$.
		\item For every $\Pi_{\beta}$ formula $\psi$ there is a $\Pi_{\alpha+\beta}$ formula $\psi_*$ such that $\mc{M}\models \psi \iff \Phi(\mc{M})\models \psi_*$.
		\item A structure $\mc{A}\in\mathbb{D}$ has Scott complexity $\Gamma_\delta$ where $\Gamma_{\alpha+\delta}$ is at least $\Pi_\gamma$ then the Scott complexity of $\Phi(\mc{A})$ is $\Gamma_{\alpha+\delta}$.
	\end{enumerate}
\end{proposition}

When applying (1), the formula $\varphi^*$ is referred to as the \textit{pull-back} of $\varphi$.
When applying (2), the formula $\psi_*$ is referred to as the \textit{push-forward} of $\psi$.

\begin{proof}
	The first two points are immediate consequences of Proposition 11 of \cite{MR23} along with the pull-back theorem of Knight, Miller, and Vanden Boom \cite{KnightMillervB}.
	
	Therefore, we focus on the third point.
	Note that the $\Gamma_\delta$ Scott sentence of $\mc{A}$ pushes forward to a $\Gamma_{\alpha+\delta}$ sentence true of $\Phi(\mc{A})$.
	This push-forward sentence, along with the $\Pi_{\gamma}$ definition of the image of $\Phi$, gives a Scott sentence for $\Phi(\mc{A})$, as no other isomorphism class in the image satisfies this sentence.
	If there were a simpler Scott sentence for $\Phi(\mc{A})$, it would pull back to a sentence simpler than $\Gamma_\delta$ that would be a Scott sentence for $\mc{A}$.
	This is an immediate contradiction, so $\Phi(\mc{A})$ has Scott complexity $\Gamma_{\alpha+\delta}$ as desired.
\end{proof}

\begin{definition}
	If $\mathbb{C}$ and $\mathbb{D}$ satisfy the above conditions for some $\gamma$ we will say that $\mathbb{C}$ $\alpha$-\textit{almost bi-interprets} $\mathbb{D}$ and that the embedding $\Phi$ is an $\alpha$-\textit{almost bi-interpretation}.
\end{definition}

These interpretations are useful because they preserve Scott rank up to a constant and predictable increase in quantifier complexity.
In this way, they are similar to the effective bi-interpretations used in \cite{MR23}.
That said, there are far more almost bi-interpretations, which makes the tool more widely applicable (which is key for our purposes).
For example, almost bi-interpretations do not need to preserve the automorphism group of the structure in the way that an effective bi-interpretation must. 
The notion of an almost bi-interpretation interacts with other already studied concepts.
For example, the fourth condition is equivalent to the descriptive set-theoretic notion of a map being \textit{faithful}, or that it brings Borel sets to sets whose $\cong$-saturation is Borel (see, for example \cite{FS89}).
The above proposition can be seen as saying that $\alpha$-almost bi-interpretations are uniformly faithful in a strong sense.
In particular, the complexity of the Borel set $\alpha$ is additively tied with the complexity of the $\cong$-saturation of its image at $\alpha+\beta$.
In practice, when providing a faithful interpretation, this sort of uniformity is often present.
This is even true for complex arguments such as Paolini and Shelah's recent construction of a faithful embedding of graphs into torsion-free Abelian groups \cite{PaoliniShelah}. 

It is worth noting that the size of the gap between the complexity of a theory and the complexity of the simplest model (i.e. the primary object of study in this paper) can be seen as an obstruction to coding structures using an $\alpha$-almost bi-interpretation.

\begin{proposition}
Let $\mathbb{G}$ be the class of graphs.
If $\mathbb{D}$ has the property that there is a countable ordinal $\beta$ such that every $\Pi_\delta$ subset of $\mathbb{D}$ has a Scott rank $\leq\delta+\beta$ structure, then $\mathbb{D}$ does not accept an $\alpha$-almost bi-interpretation from $\mathbb{G}$ for any $\alpha$.
\end{proposition}

\begin{proof}
For the sake of contradiction say that $\Phi:\mathbb{G}\to (\mathbb{D})_\alpha$ is an $\alpha$-almost bi-interpretation with $\Pi_\gamma$ image.
Let $T$ be a $\Pi_2$ definable set of graphs with models of only Scott rank $\iota$ where $\iota>\max(2+\alpha+\beta,\gamma)$.
It is possible to find such a class by the main result of \cite{HTScott}.
By Proposition \ref{prop:DinoAntonio}, $\Phi(T)$ is a $\Pi_{2+\alpha}$ definable set, and each of the structures in $\Phi(T)$ are of Scott rank $\iota+\alpha$.
However, this is a contradiction to our assumption as $\iota+\alpha>2+\alpha+\beta$.
Therefore, no such $\alpha$-almost bi-interpretation can exist.
\end{proof}

It is an immediate corollary of the above along with our main result that linear orderings do not accept an $\alpha$-almost bi-interpretation from graphs.
That being said, this was already known for general faithful Borel embeddings (see \cite{Gao01}).
The analogous result for Boolean algebras, i.e., Question \ref{question:Boolean}, would provide a new result via the above proposition.
In general, the gap between the complexity of a theory and the complexity of the simplest model may be useful as an invariant to further our understanding of strong interpretability conditions.

In Section \ref{sec:faithful} we will prove a weaker analogue of this proposition for faithful Borel embeddings: In particular, for faithful embeddings, we do prove that the bound $\beta$ is the same for every $\delta$.

\subsection{Scott Skipping at a Successor Ordinal}

To construct an example of non-trivial Scott skipping away from a limit ordinal, we first construct an example in a slightly different language.
We construct an example of Scott skipping that is a linear ordering with countably many unary predicates that is $\Pi_2$.
From there, we translate this example up to a $\Pi_4$ example that is a pure linear ordering.

Consider the following $\Pi_2$ theory $T$ over the language $\leq,\{R_i\}_{i\in\omega}$ where $\leq$ is a linear ordering and each $R_i$ is a unary predicate.
\begin{definition}\label{def:T}
$T=\{LO_{\leq},\psi,\theta\}$. $\psi$ states that any sequence of $R_i$ and $\lnot R_i$ is dense, i.e.
\[\psi:= \bigwwedge_{\sigma\in2^{<\omega}}\forall x< y ~ \exists z ~ \big( x<z<y \land \bigwedge_{1\leq i\leq \vert \sigma \vert} (\lnot)^{\sigma(i)}R_i(z)\big),\]
where $(\lnot)^{1}$ is just $\lnot$ and $(\lnot)^{0}$ is nothing at all.
$\theta$ states that no two elements have the same predicates that hold of them, i.e.
\[ \theta:= \forall x\neq y ~ \bigvvee_{i\in\omega} \lnot \big( R_i(x)\iff R_i(y) \big).\]
\end{definition}


With standard methods we can show that $T$ has models.

\begin{proposition}
$T$ is satisfiable.
\end{proposition}

%
%

We now extend $T$ to a theory $S$ over the language  $\leq,\{R_i\}_{i\in\omega}, U,V$ with two additional unary predicates $U$ and $V$.

\begin{definition}
$S=\{LO_{\leq},\chi,\psi^*,\theta^*,\eta,\nu\}$. $\chi$ says that $U$ and $V$ partition the domain,
\[ \chi:= \forall x (U(x) \vee V(x)) \text{ and } \forall x \; \neg (U(x) \wedge V(x)).\]
$\eta$ says that elements with $U$ will only have $U$ hold of them i.e.
\[ \eta:= \forall x ~ U(x)\to \bigwwedge_{i\in\omega} \lnot R_i(x).\]
$\psi^*$ is the slight variant of $\psi$ which says that for any sequence of $R_i$ and $\neg R_i$ there are densely many $V$-points satisfying those relations, i.e.,
\[\psi^* := \bigwwedge_{\sigma\in2^{<\omega}}\forall x< y ~ \exists z ~ \big( x<z<y \land V(z) \land \bigwedge_{1\leq i\leq \vert \sigma \vert} (\lnot)^{\sigma(i)}R_i(z)\big),\]
$\theta^*$ is with all quantifiers relativized to elements with $V$.
\[ \theta:= \forall x \neq y \;\; (V(x) \land V(y)) \longrightarrow ~ \bigvvee_{i\in\omega} \lnot \big( R_i(x)\iff R_i(y) \big).\]
$\nu$ insists that $U$ points are also dense, i.e.,
\[\nu:= \bigwwedge_{i\in\omega}\forall x< y ~ \exists z ~ \big(x<z<y \land U(z)\big). \]
\end{definition}

Because of $\psi^*$ and $\theta^*$, models of $S$ are models of $T$ if we only look at elements with $V$.
Unlike the elements with various $R_i$ holding of them, the $U$ elements are not distinguished from each other via some relation.
Intuitively, one may think of the $U$ elements as Dedekind cuts added to a model of $T$. It is straightforward to show how to add elements to a model of $T$ to get a model of $S$.

\begin{proposition}
$S$ is satisfiable.
\end{proposition}

%

We now show that, even though $S$ is a $\Pi_2$ theory, it does not have any model of Scott rank 2. 

\begin{proposition}\label{prop:unaryScottSkip}
Every model of $S$ has a $\Pi_4$ Scott sentence but no $\Sigma_4$ Scott sentence.
\end{proposition}
\begin{proof}
Fix a model $\mc{A} \models S$. Note that for any two elements $a \neq b$, say $a < b$, there is an element $c \in (a,b)$ with $V(c)$. We can distinguish $c$ uniquely among all other elements by the sequence $\sigma \in 2^\omega$ with $\sigma(n) = 0 $ if  $\mc{A} \models R_n(c)$ and $\sigma(n) = 1$ if $\mc{A} \models \neg R_n(c)$. Then we can distinguish $a$ from $b$ by the fact that $a$ has such an element greater than it, and $b$ has such an element less than it.
\[ \mc{A} \models \exists z\;\; V(z) \wedge z > a \wedge \bigdoublewedge_{n} (\neg)^{\sigma(i)} R_n^{\sigma(n)}(z) \]
and
\[ \mc{A} \nmodels \exists z \;\; V(z) \wedge z > b \wedge \bigdoublewedge_{n} (\neg)^{\sigma(i)} R_n^{\sigma(n)}(z) \]
and similarly
\[ \mc{A} \models \exists z \;\; V(z) \wedge z < b \wedge \bigdoublewedge_{n} (\neg)^{\sigma(i)} R_n^{\sigma(n)}(z) \]
and
\[ \mc{A} \nmodels \exists z \;\; V(z) \wedge z < a\wedge \bigdoublewedge_{n} (\neg)^{\sigma(i)} R_n^{\sigma(n)}(z). \]
Thus, if $a \neq b$, then $a \nleq_2 b$ and $b \nleq_2 a$. Thus, for tuples $\bar{a}$ and $\bar{b}$, $\bar{a}\leq_2\bar{b}\implies \bar{a} = \bar{b}$ (as otherwise, for some $i$, $a_i \leq_2 b_i$ but $a_i \neq b_i$). This implies that $\mc{A}$ has a $\Pi_4$ Scott sentence. In fact we have also shown that every model of $\mc{A}$ is rigid.

We know show that no model $\mc{A} \models S$ has a $\Sigma_4$ Scott sentence, for which it suffices to show that $\mc{A}$ has no $\Pi_3$ Scott sentence over parameters $\bar{r}$.
We first note a special property about $\Pi_1$ formulas about models of $\mc{A}$.
In particular, for any $a,b\in \mc{A}\cup\{\infty,-\infty\}$ and $\varphi\in \Pi_1$ without parameters $\mc{A}\models \varphi\Res_{(a,b)} \iff \mc{A}\models \varphi$.
As $(a,b)$ is a substructure of $\mc{A}$, the reverse direction follows at once.
To see the forward direction, it is enough to show that $(a,b)\leq_1 \mc{A}$.
Let the $\exists$-player pick $x_1<\cdots<x_n$ in $\mc{A}$.
The $\forall$-player must now pick $y_1<\cdots<y_n$ in $(a,b)$ with the same quantifier-free type up to the first $n$ formulas.
That said, by density of $R_i$, $\lnot R_i$, $U$ and $V$ there are always $y_1<\cdots<y_n$ in $(a,b)$ with any given finite quantifier free diagram.
Therefore, the $\forall$-player can always win and $(a,b)\leq_1 \mc{A}$ as desired.

It follows that any two intervals $(a,b)$ and $(b,c)$ of $\mc{A}$ have that $(a,b)\equiv_1(b,c)$.
From this, we claim that any two tuples $\bar{p}$ and $\bar{q}$ in $\mc{A}$ with the same quantifier-free type over some parameters $\bar{r}$ have $\bar{p}\bar{r}\equiv_1\bar{q}\bar{r}$.
By playing the back-and-forth game within the intervals defined by $\bar{r}$ it is enough to show that the parts of $\bar{p}$ and $\bar{q}$ intersecting each interval are 1-equivalent.
As every interval of a model of $S$ is itself a model of $S$, it is therefore sufficient to demonstrate the version of our above claim without parameters.
In other words, we show that any two tuples $\bar{p}$ and $\bar{q}$ in a model of $S$ with the same quantifier-free type have $\bar{p}\equiv_1\bar{q}$.
Without loss of generality, say that $\bar{p}=p_1<\cdots<p_k$ and $\bar{q}=q_1<\cdots<q_k$ with the convention that $p_0=q_0=-\infty$ and $p_{k+1}=q_{k+1}=\infty$.
We describe a winning strategy in the game $\bar{p}\leq_1\bar{q}$ for the $\forall$ player under this assumption and therefore show that $\bar{p}\equiv_1\bar{q}$ by symmetry.
Say the $\exists$-player plays $\bar{k}$.
Break  $\bar{k}$ into $\bar{k}_i\in (q_i,q_{i+1})$.
Let $\bar{\ell}_i\in (p_i,p_{i+1})$ be the winning response to $\bar{k}_i$ in the $(q_i,q_{i+1})\geq_1(p_i,p_{i+1})$ game.
The $\forall$-player plays $\bar{\ell}$, the tuple of all of the $\bar{\ell}_i$.
Note that, by assumption, $\bar{p},\bar{\ell}$ and $\bar{q},\bar{k}$ agree on all quantifier-free formulas up to their length regarding the unary predicates.
Furthermore, they were selected to come in the same order, so they are $0$-equivalent as desired.
Therefore, $\bar{p}\equiv_1\bar{q}$ as desired.

Consider a $U$ point $x\in \mc{A}$.
For the sake of contradiction say that $x$ has a $\Sigma_2$ definition $\mc{A}\models\exists \bar{z} ~\psi(x,\bar{z})$ with $\psi\in\Pi_1$.
Let $\bar{p}= p_1<\cdots<p_k$ witness $\mc{A}\models\psi(x,\bar{z})$.
Say that $p_0<\cdots<p_r<x<p_{r+1}<\cdots < p_{k+1}$.
By density of $U$ there is another $U$ point $y$ with $p_r<y<x$.
The tuple $\bar{p},y$ therefore has the same quantifier-free type as $\bar{p},x$.
Therefore, $\bar{p},y\equiv_1 \bar{p},x$ and so  $\mc{A}\models \psi(y,\bar{p})$.
In particular,  $\mc{A}\models\exists \bar{z} ~ \psi(y,\bar{z})$.
This means that $y$ and $x$ are in the same isomorphism orbit.
However, any $V$ element $v$ between $y$ and $x$ has a unique quantifier free type by the sentence $\theta^*$ in $S$.
Any automorphism taking $x$ to $y$ would have to move $v$ to a different $V$ element below $y$, a contradiction. As the automorphism orbit of $x$ has no $\Sigma_2$ definition over parameters, $\mc{A}$ has no $\Sigma_4$ Scott sentence as desired.
\end{proof}

For our purposes the most important aspect of the above results is that it demonstrates that a Scott rank is skipped.
This is the basis of our example in linear orderings.
We must describe a systematic way to transform linear orderings with countably many unary predicates into linear orderings.
This method will preserve the skipping of Scott ranks as it will be a 2-almost bi-interpretation. 

\begin{definition}\label{def:PhiBiInt}
Let  $\mc{L}=(L,\leq,\{S_i\}_{i\in\omega})$ be a linear ordering with countably many unary predicates. Let
\begin{itemize}
	\item for $x\in X$, let $c_i(x)$ be $2$ if $\mc{L} \models S_i(x)$ and $c_i(x) = 3$ if $\mc{L} \models \neg S_i(x)$.
	\item $\Phi(\mc{L})=\sum_{x\in L} 4 + \left( \sum_{i \in \omega} \mathbb{Q}+ c_i(x) + \mathbb{Q} + 2 + \mathbb{Q} + 3 + \mathbb{Q}\right) + 4$, a pure linear ordering.
\end{itemize}

\end{definition}

Note that this is an effective interpretation of $\Phi(\mc{L})$ in $\mc{L}$. We show that there is a $\Delta_3$ interpretation of $\mc{L}$ in $\Phi(\mc{L})$ (and hence an effective interpretation of $\mc{L}$ in $\Phi(\mc{L})_{(2)}$) and also that the former effective interpretation can be strengthened to an effective interpretation of $\Phi(\mc{L})_{(2)}$ in $\mc{L}$.

\begin{lemma}\label{lem:bi-int}
$\mc{L}=(L,\leq,\{S_i\}_{i\in\omega})$ can be interpreted in $\Phi(\mc{L})$ in a $\Delta_3$ manner. 
\end{lemma}

\begin{proof}
By definition, it is enough to show that there are $\Delta_3$ definable relations $Dom_\mc{L}$, $\sim_\mc{L}$, $\leq_\mc{L}$ and $S_{i,\mc{L}}$ such that
\[ (Dom_\mc{L},\leq_\mc{L},S_{i,\mc{L}})/\sim_\mc{L}  \cong \mc{L}. \]
The first three definitions are straightforward.
\begin{itemize}
	\item $Dom_\mc{L}$ is the set of 8-tuples $(x_1,x_2,x_3,x_4,y_1,y_2,y_3,y_4)$ such that $x_1 < x_2 < x_3 < x_4$ are consecutive and $y_1 < y_2 < y_3 < y_4$ are consecutive, $x_4$ does not have a successor, $y_1$ does not have a predecessor, and there are no four consecutive elements between $x_4$ and $y_1$. This is $\Pi_2$.
	\item $\sim_\mc{L}$ is just equality, hence is $\Delta_1$.
	\item $(x_1,x_2,x_3,x_4,y_1,y_2,y_3,y_4) <_\mc{L} (x_1',x_2',x_3',x_4',y_1',y_2',y_3',y_4')$ if and only if $y_4 < x_1'$, hence this relation is $\Delta_1$.
	\item $S_{i,\mc{L}}$ is definable in a $\Sigma_3$ way, as the set of tuples $(x_1,x_2,x_3,x_4,y_1,y_2,y_3,y_4)$ such that for some $c_0,\ldots,c_{i-1} \in \{2,3\}^i$ there are $z_0,z_1,z'$ in $(x_4,y_1)$ such that $z_1$ is the successor of $z_0$, the interval $(z_1,z')$ is isomorphic to $\mathbb{Q}$, and the interval $(x_4,z_0)$ is isomorphic to \[\mathbb{Q} + c_0 + \mathbb{Q} + 2 + \mathbb{Q} + 3 + \mathbb{Q}+ c_1 + \mathbb{Q} + 2 + \mathbb{Q}+ \cdots + \mathbb{Q} + c_{i-1} + \mathbb{Q} + 2 + \mathbb{Q} + 3 + \mathbb{Q}.\]
	To see that this is $\Sigma_3$, note that there is a $\Pi_2$ formula that says that an interval is isomorphic to $\mathbb{Q}$.
	
	\item Similarly, $\neg S_{i,\mc{L}}$ is definable in a $\Sigma_3$ way, as the set of tuples $(x_1,x_2,x_3,x_4,y_1,y_2,y_3,y_4)$ such that for some $c_0,\ldots,c_{i-1} \in \{2,3\}^i$ there are $z_0,z_1,z_2,z'$ in $(x_4,y_1)$ such that $z_1$ is the successor of $z_0$ and $z_2$ is the successor of $z_1$, the interval $(z_2,z')$ is isomorphic to $\mathbb{Q}$, and the interval $(x_4,z_0)$ is isomorphic to \[\mathbb{Q} + c_0 + \mathbb{Q} + 2 + \mathbb{Q} + 3 + \mathbb{Q}+ c_1 + \mathbb{Q} + 2 + \mathbb{Q}+ \cdots + \mathbb{Q} + c_{i-1} + \mathbb{Q} + 2 + \mathbb{Q} + 3 + \mathbb{Q}.\]
\end{itemize} 
It is easy to see that this is an interpretation of $\mc{L}$ in $\Phi(\mc{L})$ using $\Delta_3$ formulas.
\end{proof}

\begin{corollary}
	$\mc{L}$ is computably interpretable in $\Phi(\mc{L})_{(2)}$.
\end{corollary}

Note that different aspects of the interpretation are of different complexities.
In particular, the $S_i$ are $\Delta_3$ definable, but everything else has a $\Pi_2$ definition.

\begin{proposition}\label{prop:Pi4}
The image of the models of $S$ under $\Phi$ is $\Pi_4$.
\end{proposition}

\begin{proof}
We begin by defining the image of all linear orderings with unary predicates under $\Phi$.
Consider the following properties:
\begin{enumerate}
	\item For every element $z$, there are $(x_1,x_2,x_3,x_4,y_1,y_2,y_3,y_4) \in Dom_{\mc{L}}$ such that $x_1 \leq z \leq y_4$. Moreover, if $x_4 < z < y_1$ then there is some $n$ and $c_0,\ldots,c_{n-1} \in \{2,3\}^n$ and such that $z$ is contained in an initial segment of $(x_4,y_1)$ of the form
	\[\mathbb{Q} + c_0 + \mathbb{Q} + 2 + \mathbb{Q} + 3 + \mathbb{Q}+ c_1 + \mathbb{Q} + 2 + \mathbb{Q}+ \cdots + \mathbb{Q} + c_{n-1} + \mathbb{Q} + 2 + \mathbb{Q} + 3 + \mathbb{Q} + 1.\]
	This is $\Pi_4$.
	\item For all $(x_1,x_2,x_3,x_4,y_1,y_2,y_3,y_4) \in Dom_{\mc{L}}$ and every $n$, there are $c_0,\ldots,c_{n-1} \in \{2,3\}^n$ such that 
	\[\mathbb{Q} + c_0 + \mathbb{Q} + 2 + \mathbb{Q} + 3 + \mathbb{Q}+ c_1 + \mathbb{Q} + 2 + \mathbb{Q}+ \cdots + \mathbb{Q} + c_{n-1} + \mathbb{Q} + 2 + \mathbb{Q} + 3 + \mathbb{Q} + 1\]
	is an initial segment of $(x_4,y_1)$. This is also $\Pi_4$.
\end{enumerate}
To see that every such linear order is in the image of $\Phi$, note that an interval cannot have initial segments of the form 
\[\mathbb{Q} + c_0 + \mathbb{Q} + 2 + \mathbb{Q} + 3 + \mathbb{Q}+ c_1 + \mathbb{Q} + 2 + \mathbb{Q}+ \cdots + \mathbb{Q} + c_{n-1} + \mathbb{Q} + 2 + \mathbb{Q} + 3 + \mathbb{Q} + 1\]
for two different values of $c_0,\ldots,c_{n-1}$.

To shift from the image of all linear ordering to the image of $S$ under $\Phi$, we pull back the $\Pi_2$ sentences of $S$ along the interpretation from Lemma \ref{lem:bi-int}. As the interpretation is $\Delta_3$,  $\Pi_2$ sentences become $\Pi_4$.
\end{proof}

We recall the classification from \cite{McountingBF} of the $\equiv_2$-types of linear orders, which are in particular simplified in the case of structures $\Phi(\mc{L})$ since they do not contain any instances of 5 consecutive elements. First, by Lemma \ref{lem:combine-bf}, rather than considering the back-and-forth types of tuples, it suffices to consider the back-and-forth types of linear orders. Second, we reduce to the case of linear orders without endpoints: $\mc{A} \equiv_2 \mc{B}$ if and only if there are $m$ and $n$ such that $\mc{A} \cong m + \mc{A}^* + n$ and $\mc{B} \cong m + \mc{B}^* + n$ such that $\mc{A}^*$ and $\mc{B}^*$ have no first or last elements and $\mc{A}^* \equiv_2 \mc{B}^*$. Now $\mc{A}^* \equiv_2 \mc{B}^*$ if and only if they have the same number of tuples of 4 consecutive elements; and if there are only finitely many, then they have the same number of tuples of 3 consecutive elements; and if there are only finitely many, then they have the same number of 2 consecutive elements. 

Let $\mc{L}_{(2)}$ be the language of linear orders together with unary relations for the $\equiv_2$-types.

\begin{theorem}
	$\Phi$ is a computable operator mapping labeled linear orders to linear orders with the $\equiv_2$-types named (as structures in the language $\mc{L}_{(2)}$).
\end{theorem}
\begin{proof}
	It suffices to note that given $\mc{L}$ a labeled linear order, from $\mc{L}$ we can compute a ``standard copy'' of $\Phi(\mc{L})$ in which we can compute for any two elements $x < y$ (a) how many successors $x$ has, (b) how many predecessors $y$ has, and (c) how many tuples of 4 consecutive elements are in the interval $(x,y)$. For (c), the answer is either none or infinitely many; if it is none, then $x$ and $y$ are in the same copy of
	\[4 + \left( \sum_{i \in \omega} \mathbb{Q}+ c_i(x) + \mathbb{Q} + 2 + \mathbb{Q} + 3 + \mathbb{Q}\right) + 4\]
	and we can compute (d) the number of tuples of 2 and of 3 consecutive elements between $x$ and $y$. It is important here that for any $x$ and $y$, the number of such tuples is either determined by whether $S_i$ holds of $x$ for some finite number of $i$'s or is infinitely many. (This latter fact is why we use $\mathbb{Q}+ c_i(x) + \mathbb{Q} + 2 + \mathbb{Q} + 3 + \mathbb{Q}$ rather than the simpler $\mathbb{Q}+ c_i(x) + \mathbb{Q} $). Together, these allow us to compute the $\equiv_2$-types on this copy of $\Phi(\mc{L})$.
\end{proof}

\begin{corollary}
	$\Phi(\mc{L})_{(2)}$ is computably interpretable in $\mc{L}$.
\end{corollary}

It is not hard to see that the interpretation of $\mc{L}$ in $\Phi(\mc{L})_{(2)}$ and vice versa are inverse. 
This means that $\Phi$ is a 2-almost bi-interpretation.
Thus by Proposition \ref{prop:DinoAntonio}, we get:

\begin{theorem}\label{thm:skipLO}
There is a $\Pi_4$ sentence extending the theory of linear orderings, all of whose models have Scott complexity at least $\Pi_6$.
\end{theorem}

We note that this is optimal in the sense that there is no analogous result for a $\Pi_3$ sentence.
This follows directly from Theorem 2.7 of \cite{GR23}, which states that every linear ordering is $3$-below a linear ordering from a countable list, all of which have $\Pi_4$ Scott sentences.
In particular, every $\Pi_3$ sentence true of a linear ordering is also true of one with a $\Pi_4$ Scott sentence.

\begin{theorem}\cite{GR23}
The following class $K$ is $3$-universal for the class of linear orderings, in other words, every linear ordering is $3$-below one of the following linear orderings.
\[K=\{\omega+k, k+\omega^*, k+\zeta+k':k,k'\in\mathbb N\}\cup\{\omega+\omega^*\}\cup\]
\[\{(\sum_{i\in k} n_i + m_i\cdot\eta )+n_{k+1}:n_i,m_i\in\mathbb N,n_i>m_{i-1},m_{i+1}\}\cup\{k:k\in\mathbb N\}.\]
Every linear ordering in the class $K$ has a $\Pi_4$ Scott sentence.
\end{theorem}

We now return to an example promised earlier, Proposition \ref{prop:must-use-new}. We use a similar strategy where we give a $\Pi_2$ theory of linear orders with unary relations, and then produce an equivalent $\Pi_4$ theory of linear orders without any additional labels.

	\begin{proposition}\label{prop:must-use-new2}
		There is a consistent $\Pi_2$ sentence $\theta^*$ in the language $\mc{L} = \{ < \} \cup \{ R_i : i \in \omega\}$ expanding the theory of linear orders such that for any consistent $\Pi_2$ sentences $\varphi$ and $\psi$ there are $\mc{A} \models \varphi$ and $\mc{B} \models \psi$ such that $\mc{A} + 1 + \mc{B} \nmodels \theta^*$ (for any labeling of the partitioning element by the $R_i$.)
	\end{proposition}
	\begin{proof}
		Let $\theta^*$ be the theory $T$ of Definition \ref{def:T}.
		In other words, it says that the linear ordering is dense and without endpoints, that no two elements satisfy exactly the same set of relations $R_i$, and that given $x < y$ and a finite string $\sigma \in 2^{< \omega}$, there is $x < z < y$ with $\sigma(i) = 1 \Longleftrightarrow R_i(z)$. We can think of each element as being labeled by a real so that $\theta^*$ says that no two elements are labeled by the same real and that the linear order is dense simultaneously as a linear order and with the labels. (This proof, at its heart, is based on the fact that two dense $\mathbf{\Pi}^0_2$ sets intersect.)
		
		Suppose that we have two $\Pi_2$ sentences $\varphi$ and $\psi$ such that whenever $\mc{A} \models \varphi$ and $\mc{B} \models \psi$, $\mc{A} + 1 + \mc{B} \models \theta^*$. We will argue that we can build $\mc{A} \models \varphi$ and $\mc{B} \models \psi$ with elements $a_0 \in \mc{A}$ and $b_0 \in \mc{B}$ such that $a_0$ and $b_0$ satisfy the same relations $R_i$. This will yield a contradiction, as then $\mc{A} + 1 + \mc{B} \nmodels \theta^*$.
		
		We build $\mc{A} = \bigcup \mc{A}_s$ and $\mc{B} = \bigcup \mc{B}_s$ by stages as the union of finite structures. At each stage $s$ the finite structures $\mc{A}_s$ will consist of elements $a_0,\ldots,a_{k_s}$ and $\mc{B}_s$ will consist of $b_0,\ldots,b_{\ell_s}$. There will be a number $m_s$ such that we only consider the relations $R_0,\ldots,R_{m_s-1}$ at this stage, i.e., we do not determine in $\mc{A}_s$ or $\mc{B}_s$ which of the other relations hold of any elements. At each stage $s$, $a_0$ and $b_0$ will satisfy the same relations $R_i$, $i < m_s$.
		
		Fix $\mc{A}^* \models \varphi$ and $\mc{B}^* \models \psi$. At each stage $s$, we will also have embeddings $f \colon \mc{A}_s \to \mc{A}^*$ and $g \colon \mc{B}_s \to \mc{B}^*$. This is because, by choice of $\theta^*$, any finite linear with any choice of unary relations is consistent with both $\varphi$ and $\psi$ and embeds into $\mc{A}^*$ and $\mc{B}^*$ (otherwise $\mc{A}^* + 1 + \mc{B}^* \nmodels \theta^*$).
		
		Let $\varphi$ be $\bigdoublewedge_i \forall \bar{x}_i \varphi^\exists_i(\bar{x}_i)$ where each $\varphi^\exists_i(\bar{x}_i)$ is $\Sigma_1$ and let $\psi$ be $\bigdoublewedge_j \forall \bar{y}_j \psi^\exists_j(\bar{y}_j)$ where each $\psi^\exists_j(\bar{y}_j)$ is $\Sigma_1$. At stage $2s+1$ we will ensure that the first $s$ conjuncts of $\varphi$, with the universal quantifier ranging among all of the elements of $\mc{A}_{2s}$, are satisfied in $\mc{A}_{2s+1}$, and similarly at stage $2s+2$ with $\psi$ and $\mc{B}_s$.
		
		The construction is as follows. At stage $0$, let $\mc{A}_0$ consist of a single element $a_0$ and $\mc{B}_0$ consist of a single element $b_0$. Let $m_0 = 0$. We now describe what to do at stages $2s+1$; stages $2s+2$ are exactly the same, except that we exchange the roles of $\mc{A}$ and $\mc{B}$.
		
		At stage $2s+1$, we act to make $\mc{A}$ satisfy $\varphi$. Recall that the elements of $\mc{A}_{2s}$ are $a_0,\ldots,a_{k_{2s}}$. Choose an embedding $f \colon \mc{A}_{2s} \to \mc{A}^*$. Then there is some finite substructure of $\mc{A}^*$ witnessing that the first $s$ conjuncts of $\varphi$, with the universal quantifiers ranging over the images of $a_0,\ldots,a_{k_{2s}}$, are true in $\mc{A}^*$. These are finitely many $\Sigma_1$ formulas, so there are finitely many elements of $\mc{A}^*$ and relations $R_i$ that witness that they are true. Expand $\mc{A}_{2s}$ to $\mc{A}_{2s+1}$ and increase $m_{2s}$ to $m_{2s+1}$ to contain all of the elements and relations $R_i$ that witness these existential formulas (ensure that $m_{2s}<m_{2s+1}$ by adding at least one new relation). In $\mc{A}_{2s+1}$ we determine which of the relations $R_i$, $i < m_{2s+1}$, are true of each element. We do not need to add any more elements of $\mc{B}$ at this stage, but since we have increased $m_{2s}$ to $m_{2s+1}$ we must determine which of the relations $R_i$, $m_{2s} \leq i < m_{2s+1}$, are true of the elements of $\mc{B}_{2s}$. We also must take action to ensure that $a_0$ and $b_0$ satisfy the same relations. For $m_{2s} \leq i < m_{2s+1}$, set $R_i(b_0)$ if and only if $R_i(a_0)$, and set $R_i(b_j)$ for each $j > 0$. This defines $\mc{B}_{2s+1}$ and is consistent with $\psi$. Therefore, we can find a finite substructure of $\mc{B}^*$ that agrees with $\mc{B}_{2s+1}$ up to relations in $m_{2s+1}$ and define an embedding $g \colon \mc{B}_{2s+1} \to \mc{B}^*$ with this image.
		
	As indicated above we let $\mc{A} = \bigcup \mc{A}_s$ and $\mc{B} = \bigcup \mc{B}_s$.
	Note that if $R_i(a_0)$ then there is some $s$ for which $m_{2s}>i$.
	Therefore, at stage $2s$ we ensure that $R_i(b_0)$ as well.
	As the relations that hold of each element are never changed in the construction, this guarantees that $\mc{A}\models R_i(a_0) \implies \mc{B}\models R_i(b_0)$.
	A symmetrical argument gives that $\mc{A}\models R_i(a_0) \iff \mc{B}\models R_i(b_0)$, so $\mc{A}+1+\mc{B}\nmodels \theta^*$.
	However, $\mc{A}\models \varphi$ and $\mc{B}\models \psi$ by construction, a contradiction.
	\end{proof}

\usenew*

\begin{proof}
	Let $\theta^*$ be as in Proposition \ref{prop:must-use-new2}. There is a $\Pi_4$ sentence $\theta$ in the language $\{\leq\}$ such that the models of $\theta$ are exactly the images, under $\Phi$ as defined in Definition \ref{def:PhiBiInt}, of the models of $\theta^*$. We use this $\theta$ as the witness to this theorem.
	
	Given consistent $\Pi_4$ sentences $\varphi$ and $\psi$, we will show that there are models $\mc{A} \models \varphi$ and $\mc{B} \models \psi$ such that $\mc{A} + 1 + \mc{B} \nmodels \theta$. Fix $\mc{A}^* \models \varphi$ and $\mc{B}^* \models \psi$ such that $\mc{A}^* +  1 + \mc{B}^* \models \theta$ (otherwise, we are done). Then the ``1'' in this linear order must be part of some interval
	\[ \mc{I} \cong 4 + \left( \sum_{i \in \omega} \mathbb{Q}+ c_i + \mathbb{Q} + 2 + \mathbb{Q} + 3 + \mathbb{Q}\right) + 4\]
	for some $c_i = 2,3$ which we split as $\mc{I} = \mc{I}_\mc{A} + 1 + \mc{I}_\mc{B}$. Let $\mc{A}^* = \mc{A}^{**} + \mc{I}_{\mc{A}}$ and $\mc{B}^* = \mc{I}_{\mc{B}} + \mc{B}^{**}$. Then $\mc{A}^{**}$ and $\mc{B}^{**}$ are themselves images under $\Phi$. There are consistent $\Pi_4$ sentences $\varphi^*$ and $\psi^*$ such that $\mc{A} \models \varphi^*$ if and only if $\mc{A}$ is an image under $\Phi$ and $\mc{A} + \mc{I}_{\mc{A}} \models \varphi$, and $\mc{B} \models \psi^*$ if and only if $\mc{B}$ is an image under $\Phi$ and $\mc{I}_{\mc{B}} + \mc{B} \models \psi$. (To find $\varphi^*$, replace a quantifier $\forall x \chi(x)$ in $\varphi^*$ with $\left(\bigdoublewedge_{a \in \mc{I}_{\mc{A}}} \chi(a) \right) \wedge \left(\forall x \chi(x)\right)$ and similarly for existential quantifiers.) By Proposition \ref{prop:DinoAntonio} there are $\Pi_2$ sentences $\varphi^\dagger$ and $\psi^\dagger$ such that $\mc{L} \models \varphi^\dagger$ if and only if $\Phi(\mc{L}) \models \varphi^*$, and $\mc{L} \models \psi^\dagger$ if and only if $\Phi(\mc{L}) \models \psi^*$. By choice of $\theta^*$ we can choose $\mc{M} \models \varphi^\dagger$ and $\mc{N} \models \psi^\dagger$ such that $\mc{M} + 1 + \mc{N} \nmodels \theta^*$ for any choice of relations holds of the separator $1$; in particular, below we will let $S_i$ hold of the separator if and only if $c_i = 2$. Then $\Phi(\mc{M}) \models \varphi^\dagger$ and $\Phi(\mc{N}) \models \psi^\dagger$ so that $\Phi(\mc{M}) + \mc{I}_{\mc{A}} \models \varphi$ and $\mc{I}_{\mc{B}} + \Phi(\mc{N}) \models \psi$. But 
	\[ \Phi(\mc{M}) + \mc{I}_{\mc{A}} + 1 + \mc{I}_{\mc{B}} + \Phi(\mc{N}) = \Phi(\mc{M}) + \mc{I} + \Phi(\mc{N}) = \Phi(\mc{N} + 1 + \mc{M}).\]
	Since $\mc{N} + 1 + \mc{M} \nmodels \theta^*$, $\Phi(\mc{N} + 1 + \mc{M}) \nmodels \theta$. Thus, letting $\mc{A} = \Phi(\mc{M}) + \mc{I}_{\mc{A}}$ and $\mc{B} = \mc{I}_{\mc{B}} + \Phi(\mc{N})$, we have $\mc{A} \models \varphi$ and $\mc{B} \models \psi$ such that $\mc{A} + 1 + \mc{B} \nmodels \theta$.
\end{proof}

\section{Changing the location of examples for linear orderings}\label{UpTheHeirarcy}

We describe below a general method of taking a construction of a linear ordering and iterating it throughout the hyperarithmetic hierarchy. 
This method generalizes many tools already seen in the literature, such as Lemmas 7.2 and 7.3 in \cite{GHKMMS}, Proposition 4.8 of \cite{Ash91}, Lemma 3.18 of \cite{CCGH} and Lemma 3.5 of \cite{GR23}.
For our purposes, we will use this method to generalize Theorem \ref{thm:skipLO} and get a $\Pi_\beta$ formula where all of the models are linear orderings with Scott complexity at least $\Pi_{\beta+2}$ for various values of $\beta$.

Our goal then is to set up an $\alpha$-almost bi-interpretation from linear orderings to the $\alpha$-canonical jump of a suitable subclass of linear orderings to take Theorem \ref{thm:skipLO} up the hierarchy.
This is done in two steps.
First, we show how to iterate exactly one jump up the hierarchy, and then we show how to iterate $2\alpha$ jumps for every $\alpha$.

All linear orderings only have countably many 1 back-and-forth types.
There is a nice presentation for the jump of linear orderings first seen in \cite{McountingBF} Section 4.1.

\begin{proposition}
Given a linear ordering $\mc{L}$, the jump $\mc{L}_{(1)}$ is given, up to effective bi-interpretation, by adding the predicates $first(x)$, $last(x)$ and $succ(x,y)$ which state that an element is the first in the ordering, the last in the ordering and that two elements have no elements in between them.
\end{proposition}

It is convenient to introduce some important helper predicates before we move on.

\begin{definition}\label{def:Helpers}
In the below definitions $x,y\in\mc{L}$, a linear ordering.
	\begin{itemize}
		\item $\mc{L}\models x\sim_1 y$ if $x$ and $y$ are finitely far apart; this is a $\Sigma_2$ formula that defines a convex equivalence relation (see: \cite{AR20}).
		\item $\mc{L}\models BS_{\leq n}(x)$ if $\#[x]_{\sim_1}\leq n$. This is a $\Pi_2$ formula (see \cite{CCGH}).
	\end{itemize}
\end{definition}

\begin{proposition}\label{prop:etaBiint}
There is an effective $1$-almost bi-interpretation given by
\[L\mapsto \big((\eta+2+\eta)\cdot L\big)_{(1)}.\]
\end{proposition}

\begin{proof}
First, we note that the mapping is effective.
As a pure linear ordering $(\eta+2+\eta)\cdot L$ is computable in $L$, one just has to replace each element of $L$ with a computable copy of $\eta+2+\eta$.
Therefore, we only need to check that the predicates $first(x)$, $last(x)$, and $succ(x,y)$ can also be effectively described.
$first(x)$ and $last(x)$ never hold of any element in $(\eta+2+\eta)\cdot L$ as it never has a first or last element.
$succ(x,y)$ should hold of exactly the elements represented by the "2" in each copy of $\eta+2+\eta$.
Notably, $succ(x,y)$ never holds between elements in different copies of $\eta+2+\eta$ within $(\eta+2+\eta)\cdot L$.
This means that one can create a copy of $((\eta+2+\eta)\cdot L\big)_{(1)}$ from $L$ by fixing a computable copy of $(\eta+2+\eta,succ)$, taking the product with $L$ to find the ordering and adding no additional elements in $succ$.

To reverse the interpretation, one must find $\Sigma_1$ predicates to define the domain and relation of $L$ within $((\eta+2+\eta)\cdot L\big)_{(1)}$.
The domain can be defined as $\{x\vert \exists y ~ succ(x,y)\}$.
As there is exactly one of these elements in each copy of $\eta+2+\eta$ and these copies are in the ordering of $L$, with this domain, one can interpret $\leq_{L}=\leq_{((\eta+2+\eta)\cdot L)_{(1)}}$.

It is straightforward to check that these maps are inverses of each other.
To finish the proof that we have an effective $1$-almost bi-interpretation, we need to observe that we can define the image.
We claim that the image is $\Pi_4$ definable in the language of linear orderings.
The following is the claimed definition:
\begin{enumerate}
	\item $\forall x ~ BS_{\leq2}(x)$,
	\item $\forall x,y, ~ succ(x)=y \to (\exists z,w ~ (z,x)\cong(y,w)\cong\eta)$,
	\item $\forall x \exists y<x \forall v,w ~ y<v<w<x \to \lnot succ(v)=w$,
	\item $\forall x \exists y>x \forall v,w ~ y>v>w>x \to \lnot succ(v)=w$,
	\item $\forall x ~ BS_{\leq1}(x)\to  ( (\exists y,z ~ succ(y)=z \land (x,y)\cong \eta) \lor(\exists y,z ~ succ(y)=z \land (z,x)\cong \eta))$.
\end{enumerate}
Given the complexity of $BS_{\leq n}$ and that $\eta$ has a $\Pi_2$ Scott sentence, all of the listed requirements are $\Pi_4$.
We now show that they define the set of linear orderings of the form $(\eta+2+\eta)\cdot L$.
Consider the set of successor pairs $T$ in an ordering $\mc{K}$ that satisfies the above properties.
$T/\sim_1$ is a linear ordering and we claim that $\mc{K}\cong (\eta+2+\eta)\cdot (T/\sim_1)$.
Let $x\sim_1y\in T$.
Consider the set of elements $B_{x,y}$ below $x$ such that $z\in B_{x,y}$ if and only if $(z,x)\cong\eta$ and $z$ is not part of a successor pair.
By Property 2, $B_{x,y}$ is not empty, and by Property 3, it has no least element.
Note that $B_{x,y}$ cannot have a maximal element as copies of $\eta$ always have end segments isomorphic to $\eta$.
In fact, if $z\in B_{x,y}$ so is any $w$ with $z<w<x$. 
Therefore, $B_{x,y}$  is a convex suborder of $\mc{K}$, and it is isomorphic to $\eta$.
By a completely symmetric argument, there is a set $A_{x,y}$ above $y$ that is a convex suborder of $\mc{K}$ and is isomorphic to $\eta$.
By Property 5, the sets $B_{x,y}$ and $A_{x,y}$ cover all elements that are not part of a successor pair.
Because $\eta$ has no successor pairs, an element not part of a successor pair belongs to at most both a $B_{x,y}$ and $A_{x',y'}$ (but not more than one of each).
In the case that $B_{x,y}\cap A_{x',y'}\neq \emptyset$, then $B_{x,y}=A_{x',y'}$ as we have that $(y,x')\cong\eta+1+\eta\cong \eta$.
In this case, pick an arbitrary cut to express $(y,x')$ as a sum of two copies of $\eta$ denoted $\eta_1+\eta_2$.
Let  $\hat{B}_{x,y}=\eta_1$ and $\hat{A}_{x',y'}=\eta_2$.
If we are not in this case, let $\hat{B}_{x,y}=B_{x,y}$ and $\hat{A}_{x',y'}=A_{x,y}$.
Now, each successor pair $x,y$ has been assigned a convex copy of $\eta$ above and below it, and moreover, these copies are disjoint.
In other words, $\mc{K}\cong(\eta+2+\eta)\cdot (T/\sim_1)$ as desired.
\end{proof}

We should note that in the above proof, the $\Pi_4$ definition of the image of multiplying by $\eta+2+\eta$ cannot be improved.
One can check that $(\eta+2+\eta)\cdot\eta \geq_4 (\eta+2+\eta)\cdot\eta+1+(\eta+2+\eta)\cdot\eta$.
Note that the structure $(\eta+2+\eta)\cdot\eta+1+(\eta+2+\eta)\cdot\eta$ fails to have Properties 3 and 4 from the above proof, so it is not the product of $\eta+2+\eta$ and some other linear ordering.
Thus, there is no $\Sigma_4$ definition of the image of multiplying by $\eta+2+\eta$, and $(\eta+2+\eta)\cdot\eta$ has no $\Pi_3$ Scott sentence.
This is notable because one may naively expect the image to be $\Pi_3$ and for $(\eta+2+\eta)\cdot\eta$ to have a $\Pi_3$ Scott sentence given the fact that this is what would happen if there were an actual effective bi-interpretation with the jump.
One must be careful when calculating the complexity of the image of an almost bi-interpretation; it may not line up with naive intuitions.

If $\mathbb{C}$ is the $\Pi_4$ class of structures defined in Theorem \ref{thm:skipLO} we can consider 
\[\mathbb{C}_n=\{L\vert \exists K\in\mathbb{C} ~ L\cong(\eta+2+\eta)^n\cdot K\}.\] 
Repeated application of the above proposition along with Proposition \ref{prop:DinoAntonio} gives that this class has structures of Scott complexity at least $\Pi_{n+6}$.
Furthermore, repeated use of Proposition \ref{prop:DinoAntonio} gives that $\mathbb{C}_n$ is a $\Pi_{n+4}$ class.
In other words, we can witness Scott skipping at any finite level greater than or equal to 4 just by using the above result.
To push into the transfinite, we need a slightly different construction.

A key result is the following, from \cite{GR23}, which helps us understand the back-and-forth relations between linear orderings that are products of $\zeta^\alpha$ with another linear ordering.

\begin{lemma}[\cite{GR23}]\label{lem:timesZ}
For the sake of organization, consider the following ordinal indexed propositions.
\begin{enumerate}\tightlist
\item[($A_\alpha$)] For any $\K$ and $\L$, $\zeta^\alpha\cdot \K
    \leq_{2\alpha}\zeta^\alpha\cdot \L$.
\item[($B_\alpha$)] For any $\K$ and $\L$ with $|\K|\geq|\L|$, $\zeta^\alpha\cdot \K
    \leq_{2\alpha+1}\zeta^\alpha\cdot \L$.
\item[($C_\alpha$)] For any $\K$ and any $\L$ without a last element,
    $\zeta^\alpha\cdot(\omega+\K)\leq_{2\alpha+2} \zeta^\alpha\cdot(\omega+\L)$.
\end{enumerate}
For any countable ordinal $\alpha$, $A_\alpha$, $B_\alpha$ and $C_\alpha$ are true.
\end{lemma}

Note that $\zeta^\alpha$ has a unique final segment $\xi_\alpha$ and that $\zeta^\alpha=\xi_\alpha^*+\xi_\alpha$ and $\xi_{\alpha+1}=\xi_\alpha+\zeta^\alpha\cdot\omega$.
We also need the notion of the generalized block relation $\sim_\alpha$.
We follow the definitions in \cite{AR20} and note their result that $\sim_\alpha$ is $\Sigma_{2\alpha}$ definable.

\begin{proposition}\label{prop:zetaAlphaTechnical}
The class of linear orderings of the form $\zeta^{\alpha}\cdot K$:
\begin{enumerate}
	\item Is definable by a $\Pi_{2\alpha+1}$ sentence.
	\item Has $2\alpha$-equivalence classes  given by products of the $2\alpha$-equivalence classes for intervals.
	\item Has possible $2\alpha$-equivalence classes for intervals given by isomorphism classes of proper intervals of $\zeta^{\alpha}$, a single class for all proper intervals that span multiple copies of  $\zeta^{\alpha}$, a single class for the intervals with no lower bound and a single class for the intervals with no upper bound. 
	\item Has only countably many back-and-forth $2\alpha$-equivalence classes for tuples.
\end{enumerate}

\end{proposition}

\begin{proof}
To demonstrate the first claim, we define
\[\psi= ~ \forall x \bigwwedge_{\delta<\alpha} \bigwwedge_{n\in\omega} (\exists
y\ S_\delta^n(x)=y)\land (\exists y\ S_\delta^n(y)=x).\]

Here, $S_\delta^n(x)=y$ is shorthand for saying that $y$ is the $n$th $\delta$-successor of $x$.
In other words,
\[ \exists z_0\cdots z_{n} ~ x=z_0
\land y=z_n \land \bigwedge_{i<n} z_i\not\sim_\delta z_{i+1} \land \forall w~
z_i<w<z_{i+1} \to (w\sim_\delta z_i \lor w\sim_\delta z_{i+1}).\] 
Overall, this is $\Sigma_{2\delta+2}$.
This gives that $\psi$ is, at worst, $\Pi_{2\alpha+1}$.
Furthermore, $\psi$ guarantees exactly that every $\sim_\alpha$ equivalence class is isomorphic to $\zeta^\alpha$.
This is the same thing as saying that $L=\zeta^\alpha\cdot (L/\sim_{\alpha})$.
Therefore, any and all $L\models\psi$ are of the form $L=\zeta^\alpha\cdot K$ for some $K$.

The second claim is exactly demonstrated for all linear orderings in \cite{McountingBF} Section 4.1.
The fourth claim follows from the second and third claims, so all we must do is demonstrate the third claim.
The third claim can be rephrased as stating:
\begin{enumerate}
	\item[a)]  that each proper interval of $\zeta^{\alpha}$ has its isomorphism class determined by its $2\alpha$-type
	\item[b)] that for any (possibly empty) $\mc{L}$ and $\mc{K}$
	\begin{enumerate} [leftmargin=1.5cm] 
		\item[i)] $\xi_\alpha + \zeta^\alpha \cdot \mc{L} \equiv_{2\alpha} \xi_\alpha + \zeta^\alpha \cdot \mc{K}$
		\item[ii)] $\xi_\alpha + \zeta^\alpha \cdot \mc{L} +\xi_\alpha^* \equiv_{2\alpha} \xi_\alpha + \zeta^\alpha \cdot \mc{K}+\xi_\alpha^*$
		\item[iii)] $\zeta^\alpha \cdot \mc{L} +\xi_\alpha^* \equiv_{2\alpha} \zeta^\alpha \cdot \mc{K}+\xi_\alpha^*$.
	\end{enumerate}
\end{enumerate}
To see item a), note that every proper interval of $\zeta^{\alpha}$ has Hausdorff rank strictly less than $\alpha$.
Therefore, by \cite{AR20} Theorem 6, it has Scott complexity strictly less than $\Pi_{2\alpha}$.
In particular, its isomorphism type is determined by its $2\alpha$-type.

Item b) takes a bit more work.
First, note that item ii) follows from item i) if you play on the last summand isomorphically.
Also, item iii) follows from item i) by reversing the orderings.
In particular, we only need to show item i).
By transitivity of the back-and-forth relations, we only need to show that $\xi_\alpha \equiv_{2\alpha} \xi_\alpha + \zeta^\alpha \cdot \mc{K}$ for all $\mc{K}$.

In the case that $\alpha$ is a limit ordinal for every $\gamma<\alpha$, $\xi_\alpha= \xi_\gamma + \zeta^\gamma\cdot(\omega+\zeta\cdot\xi_\alpha)$.
By $C_\gamma$,
\[  \xi_\gamma + \zeta^\gamma\cdot(\omega+\zeta\cdot\xi_\alpha) \equiv_\gamma \xi_\gamma + \zeta^\gamma\cdot(\omega+\zeta\cdot\xi_\alpha+\zeta^\lambda\cdot \mc{K}).\]
Therefore, for all $\gamma<\lambda$, $\xi_\alpha \equiv_{\gamma} \xi_\alpha + \zeta^\alpha \cdot \mc{K}$, or what is the same $\xi_\alpha \equiv_{2\alpha} \xi_\alpha + \zeta^\alpha \cdot \mc{K}$.

If $\alpha=\beta+1$, we need only show that
\[ \xi_\beta+\zeta^\beta\cdot\omega \equiv_\alpha  \xi_\beta+\zeta^\beta\cdot(\omega+\zeta\cdot \mc{K}). \]
However, this follows immediately from $C_\beta$.
\end{proof}

\begin{proposition}\label{prop:zetaBiint}
Let $X$ be an oracle that can compute a presentation of $\alpha$. For any $\alpha$, there is an $X$-effective $2\alpha$-almost bi-interpretation given by
\[L\mapsto \big(\zeta^\alpha\cdot L\big)_{(2\alpha)}.\]
\end{proposition}

\begin{proof}
We begin by demonstrating that this is an effective map.
By Proposition \ref{prop:zetaAlphaTechnical}, to describe the $2\alpha$ canonical jump of $\zeta^\alpha\cdot L$, it is enough to label intervals with their isomorphism type if they are properly contained in a copy of $\zeta^\alpha$ and identify pairs of points that appear in distinct copies of $\zeta^\alpha$.
As long as $\alpha$ is $X$-computable, it is straightforward to make a copy of $\zeta^\alpha$ where the isomorphism type of each interval is also $X$-computable.
In fact, the standard construction of $\zeta^\alpha$ that we call $\mathbb{Z}^\alpha$ has this property.
Given $L$, map it to $\sum_{x\in L} \mathbb{Z}^\alpha_x$ where the relations labeling the isomorphism type of intervals of $\zeta^\alpha$ are exactly those within each  $\mathbb{Z}^\alpha_x$ and the relation that states that the elements are in distinct copies of $\zeta^\alpha$ holds exactly of elements from $\mathbb{Z}^\alpha_x$ and $\mathbb{Z}^\alpha_y$ for $x\neq y$.
This is a copy of $\big(\zeta^\alpha\cdot L\big)_{(2\alpha)}$ by construction.

Now, we define the inverse interpretation.
Let $R_\infty(x,y)$ be the relation that states that $x$ and $y$ are in different copies of $\zeta^\alpha$.
It is immediate that $L\cong(\zeta^\alpha\cdot L)/\lnot R_\infty$.
This is an effective interpretation, and it is straightforward to check that it forms an $X$-computable $2\alpha$-almost bi-interpretation with the map described above.
\end{proof}

\begin{corollary}
There is a $\Pi_\beta$ sentence where every structure that satisfies it is a linear ordering with Scott complexity at least $\Pi_{\beta+2}$ as long as $\beta$ is at least four away from a limit ordinal.
\end{corollary}

\begin{proof}
Let $\mathbb{C}$ be the $\Pi_4$ class of structures defined in Theorem \ref{thm:skipLO}. 
Write $\beta=2\delta+4$ or  $\beta=2\delta+5$.
In the first case, let
\[\mathbb{C}_\beta=\{L\vert \exists K\in\mathbb{C} ~ L\cong\zeta^\delta\cdot K\}.\] 
In the second case, let
\[\mathbb{C}_\beta=\{L\vert \exists K\in\mathbb{C} ~ L\cong\zeta^\delta\cdot(\eta+2+\eta)\cdot K\}.\] 
By Proposition \ref{prop:etaBiint} and Proposition \ref{prop:zetaBiint} along with Proposition \ref{prop:DinoAntonio}, $\mathbb{C}_\beta$ has structures of Scott complexity at least $\Pi_{\beta+2}$.
Proposition \ref{prop:DinoAntonio} gives that $\mathbb{C}_n$ is a $\Pi_{\beta}$ class.
\end{proof}

\section{Scott spectral gaps and faithful Borel reductions}\label{sec:faithful}

In this section, we will show that if a class of structures is faithfully Borel complete, then it has unbounded Scott spectral gaps.

\begin{definition}
	Let $T$ be a sentence of $\mc{L}_{\omega_1,\omega}$. We say that $T$ has \textit{bounded Scott spectral gaps} if there is some function $D:\omega_1\to\omega_1$ such that for every $\alpha\in\omega_1$ if $T'$ is a $\Pi_\alpha$ extension of $T$, there there is a model $\mc{A}$ with $\mc{A}\models T'$ and $\SR(\mc{A})\leq D(\alpha)$.
\end{definition}

If $T$ does not have bounded Scott spectral gaps, we say that it has \textit{unbounded Scott spectral gaps}.
The main result of this paper can be phrased in the above framework as follows.
The function $D(\alpha)=\alpha+3$ witnesses that $T=LO$ has bounded Scott spectral gaps, and this contrasts with graphs that have unbounded Scott spectral gaps.\footnote{Linear orders have constantly bounded Scott spectra gaps, in the sense that $D(\alpha) = \alpha + \beta$. We do not know whether there is a theory with bounded Scott spectral gaps, but without a constant bound.}

Recall that a faithful Borel embedding is one where the $\cong$-saturation of the image is Borel, or equivalently, the $\cong$-saturation of the image of any Borel set is Borel. We show that faithful Borel embeddings preserve the property of having unbounded Scott spectral gaps.

\begin{theorem}
	Let $S$ and $T$ be sentences (possibly in different languages). If $\Phi:\Mod(T)\to \Mod(S)$ is a faithful Borel embedding and $T$ has unbounded Scott spectral gaps, then so does $S$.
\end{theorem}

\begin{proof}
	To prove the theorem, we must understand the way that $\Phi$ changes the Scott rank of the models it acts on.
	
	\begin{claim}\label{claim:largeScott}
		For any Borel embedding $\Phi:\Mod(T)\to \Mod(S)$ there is a function $G:\omega_1\to\omega_1$ such that if $\mc{A}\models T$ and $\SR(\mc{A})\geq G(\alpha)$, then $\SR(\Phi(\mc{A}))>\alpha$.
	\end{claim}
	
	\begin{proof}
		Let $U_\beta:=\{\B : \B\models S \text{ and } \SR(\B)\leq \beta\}$. 
		Note that for any $\beta$, $U_\beta$ is a Borel set (see for example \cite{MBook} Lemma II.67).
		This means that for all $\beta$, $\Phi^{-1}(U_\beta)$ is also Borel.
		Say that $\Phi^{-1}(U_\beta)$ contains models of arbitrarily high Scott rank.
		This is equivalent to saying that the isomorphism relation restricted to $\Phi^{-1}(U_\beta)$ is not Borel (see for example \cite{GaoBook} Chapter 12).
		However, note that the isomorphism relation on $U_\beta$ is equivalent to the back-and-forth relation $\equiv_{\beta+1}$, so it is Borel.
		This is a contradiction as $f$ is a Borel embedding.
		Therefore, $G(\alpha)=\sup\{\SR(\mc{A}) : \SR(\Phi(\mc{A})) \leq\alpha\}$ is a well defined function with the desired properties.
	\end{proof}

	\begin{claim}\label{claim:uniformlyFaitful}
		For any faithful Borel embedding $\Phi:\Mod(T)\to \Mod(S)$ there is a function $H:\omega_1\to\omega_1$ such that if $A \subseteq \Mod(T)$ is a $\mathbf{\Pi}^0_\alpha$ set then $\Phi(A)$ is a $\mathbf{\Pi}^0_{H(\alpha)}$ set.
	\end{claim}
	
	\begin{proof}
		Because $\Phi$ is faithful, $\Image(\Phi)$ is Borel. Say that $\Phi$ and $\Image(\Phi)$ are lightface $\Delta^1_1$ relative to $X \in 2^\omega$.
		
		Given $Y \in 2^\omega$, let $\mathcal{H}^Y$ denote the Harrison ordering relative to $Y$.
		We define $\mathbb{B}_\beta$ to be the set of linear orderings $L$ (in the Polish space of linear orderings) such that there are:
		\begin{itemize}
			\item a (code for) a $\mathbf{\Pi}^0_\beta$ set $A \subseteq \Mod(T)$ and
			\item $\alpha \in \mathcal{H}^{A,X}$, an element of the Harrison linear order relative to $A$ and $X$,
		\end{itemize}
		such that
		\begin{itemize}
			\item $L=\mathcal{H}^{A,X}_{\leq \alpha}$, the restriction to the elements below $\alpha$ in $\mathcal{H}^{A,X}$, and
			\item for all $y \in \omega$ a (lightface) $\Pi^0_\alpha(A,X)$ code for a set $Y \subseteq \Mod(S)$ we have $\Phi(A)\neq Y$.
		\end{itemize}
		In the above, we equivocate between a code for $A$ and $A$ itself. We claim that for each $\beta$, $\mathbb{B}_\beta$ is
		\begin{enumerate}
			\item contained in the well-orders and
			\item $\mathbf{\Sigma}_1^1$.
		\end{enumerate}
		
		To see (1), we will first show that for each (code for) a $\mathbf{\Pi}^0_\beta$ set $A \subseteq \Mod(T)$, we have that $\Phi(A)$ is $\Delta_1^1(A,X)$. To show this we show that $\Phi(A)$ is both $\Sigma_1^1(A,X)$ and $\Pi_1^1(A,X)$.
		It is $\Sigma_1^1(A,X)$ because, by definition, it is the image of an $\Delta^1_1(A)$ set under a $\Delta^1_1(X)$ map.
		To see that it is $\Pi_1^1(A,X)$, consider $A^c$.
		Like $A$, $A^c$ is $\Delta^1_1(A)$, so $\Phi(A^c)$ is also $\Sigma_1^1(A,X)$.
		Note that $\Phi(A^c) = \Image(\Phi) - \Phi(A)$.
		In other terms $\Phi(A)^c = \Image(\Phi)^c \cup \Phi(A^c)$.
		$\Image(\Phi)^c$, like $\Image(\Phi)$, is $\Delta^1_1(X)$ so, in particular, $\Sigma_1^1(A,X)$.
		This means that $\Phi(A)^c$ is $\Sigma_1^1(A,X)$, and $\Phi(A)$ is $\Pi_1^1(A,X)$, as desired.
		
		Now given that $\Phi(A)$ is $\Delta^1_1(A,X)$, $\Phi(A)$ will be $\Pi^0_\gamma(A,X)$ for some $\gamma$ in the well-founded part of $\mathcal{H}^{A,X}$.
		So, $\Phi(A)$ will have a $\Pi^0_\gamma(A,X)$ code for every $\alpha \geq \gamma$, including every $\alpha$ in the ill-founded part of $\mathcal{H}^{A,X}$.
		Thus if $\alpha \in \mc{H}^{A,X}$ is such that $\Phi(A)$ has no $\Pi^0_\alpha(A,X)$ code, then $\alpha < \gamma$ must be in the well-founded part of $\mc{H}^{A,X}$. Hence, $L$ is a well-order.
		
		We now show (2).
		This comes down to analyzing the syntax in the definition of $\mathbb{B}_\beta$ term by term.
		In particular, we need to show that everything after the existence of $A$ is a $\mathbf{\Sigma}_1^1$ condition (in fact, we show it is $\Sigma_1^1(A,X,\beta)$).
		Checking that $A$ is a $\mathbf{\Pi}^0_\beta$ code is $\Delta^1_1(A,\beta)$; it is equivalent to checking that a coded tree has rank less than or equal to $\beta$.
		The quantifications over $\alpha$ and $y$ are of no concern as they only range over natural numbers. The suborder $\mathcal{H}^{A,X}_{\leq \alpha}$ is computable in a copy of $\mathcal{H}^{A,X}$, which is itself computable from $A$ and $X$, and it is easy to check whether $L$ is equal to it.
		Testing that $y$ is a $\Pi^0_\alpha(A,X)$ code is $\Delta^1_1(A,X)$ as $\alpha$ is an $(A,X)$ pseudo-ordinal.
		Finally, checking that $\Phi(A) \neq Y$ comes down to asserting the existence of an element in one set and not the other.
		However, each of these sets is $\Delta^1_1(A,X)$ (the first by the analysis for (1) and the second by construction), so this is a $\Sigma_1^1(A,X)$ condition, as desired.
		
		\medskip
		
		With (1) and (2) established, we now demonstrate the claim.
		For each $\mathbb{B}_\beta$ it follows from $\mathbf{\Sigma}_1^1$ bounding (see, e.g., \cite{MBook} Theorem VI.13) that there is some bound $H(\beta)$ to the ordinals appearing in $\mathbb{B}_\beta$.
		As $H(\beta)$ is not in $\mathbb{B}_\beta$, every $\mathbf{\Pi}^0_\beta$ set $A$ has a $\Pi^0_{H(\beta)}(A,X)$ set $Y$ such that $\Phi(A) = Y$.
		In less verbose terms, this means that $\Phi(A)$ is always $\mathbf{\Pi}^0_{H(\alpha)}$, as desired.
	\end{proof}
	
	With the two claims proven, we are now ready to prove the theorem.
	Given a faithful Borel embedding $\Phi:\Mod(T)\to \Mod(S)$, we let $G,H:\omega_1\to\omega_1$ be defined as above.
	By assumption, $T$ has unbounded Scott spectral gaps.
	For a contradiction, we assume that $S$ has bounded Scott spectral gaps witnessed by $D:\omega_1\to\omega_1$.
	As $T$ has unbounded Scott spectral gaps, there is some $\alpha$ such that for any $\beta$ there is a $\Pi_\alpha$ extension $T_\beta$ of $T$ such that $\mathcal{M}\models T_\beta\implies \SR(\mathcal{M})\geq \beta$.
	Let $S_\beta=\Phi(T_\beta)$; note that for any $\beta$, $S_\beta$ is $\mathbf{\Pi}^0_{H(\alpha)}$, so it must have a model of Scott rank at most $D(H(\alpha))$.
	Take $\beta=G(D(H(\alpha)))$.
	By definition of $T_{G(D(H(\alpha)))}$ and $G$, this means that every structure in $S_{G(D(H(\alpha)))}$ has Scott rank strictly greater than $D(H(\alpha))$, a contradiction.
\end{proof}

From our Theorem \ref{thm:SmallGaps} combined with the analysis of this section, we get another proof of the following fact that was first proved by Gao \cite{Gao01}.

\begin{corollary}[Gao]
	Linear orderings are not faithfully Borel complete.
\end{corollary}

On the other hand, since torsion-free abelian groups are faithfully Borel complete as shown by Paolini and Shelah \cite{PaoliniShelahFaithful}, we cannot bound their Scott spectra gaps.

\begin{corollary}
	Torsion-free Abelian groups do not have bounded Scott spectral gap.
\end{corollary}

\bibliographystyle{alpha}
\bibliography{references}

\end{document}